\documentclass[a4paper]{article}
\usepackage{amsthm,amssymb,amsmath,enumerate,graphicx,epsf}

\newcommand{\COLORON}{0}
\newcommand{\NOTESON}{0}
\newcommand{\Debug}{0}
\usepackage{bbold}
\usepackage[usenames]{color} 
\usepackage{amsthm,amssymb,amsmath,bbm,enumerate,graphicx,epsf,stmaryrd}
\usepackage[bookmarks, colorlinks=false, breaklinks=true]{hyperref} 

\newcommand{\comment}[1]{}
\newcommand{\COMMENT}[1]{}

\definecolor{darkgray}{rgb}{0.3,0.3,0.3}
\newcommand{\defi}[1]{{\color{darkgray}\emph{#1}}}

\newcommand{\acknowledgements}{\section*{Acknowledgements}}

\comment{
	\begin{lemma}\label{}	
\end{lemma}
\begin{proof}

\end{proof}

\begin{theorem}\label{}
\end{theorem} 
\begin{proof} 	

\end{proof}

}

\newtheorem{proposition}{Proposition}[section]

\newtheorem{theorem}[proposition]{Theorem}
\newtheorem{corollary}[proposition]{Corollary}

\newtheorem{lemma}[proposition]{Lemma}

\newtheorem{conjecture}{{Conjecture}}[section]

\newtheorem{problem}[conjecture]{{Problem}}

\newtheorem{examp}[proposition]{Example}

\newcommand{\kreis}[1]{\mathaccent"7017\relax #1}

\newcommand{\FIG}{0}

\ifnum \NOTESON = 1 \newcommand{\note}[1]{ 

\hspace*{-30pt}
	{\color{blue}  NOTE: \color{Turquoise}{\small  \tt \begin{minipage}[c]{1.1\textwidth}  #1 \end{minipage} \ignorespacesafterend }} 
	
	}
\else \newcommand{\note}[1]{} \fi

\newcommand{\afsubm}[1]{ \ifnum \Debug = 1 {\mymargin{#1}}
\fi} 

\ifnum \Debug = 1 
\else  \fi

\ifnum \FIG = 1 \newcommand{\fig}[1]{Figure ``{#1}''}
\else \newcommand{\fig}[1]{Figure~\ref{#1}} \fi

\ifnum \FIG = 1 
\else  \fi

\ifnum \Debug = 1 \usepackage[notref,notcite]{showkeys}
\fi

\ifnum \COLORON = 0 \renewcommand{\color}[1]{}
\fi

\newcommand{\showFig}[2]{
   \begin{figure}[htbp]
   \centering
   \noindent
   \epsfbox{#1.eps}
   \caption{\small #2}
   \label{#1}
   \end{figure}
}

\newcommand{\D}{\ensuremath{\mathbb D}}
\newcommand{\N}{\ensuremath{\mathbb N}}
\newcommand{\R}{\ensuremath{\mathbb R}}
\newcommand{\C}{\ensuremath{\mathbb C}}
\newcommand{\Z}{\ensuremath{\mathbb Z}}

\newcommand{\BS}{\ensuremath{\mathbb S}}

\newcommand{\ca}{\ensuremath{\mathcal A}}
\newcommand{\cb}{\ensuremath{\mathcal B}}
\newcommand{\cc}{\ensuremath{\mathcal C}}

\newcommand{\cs}{\ensuremath{\mathcal S}}

\newcommand{\cu}{\ensuremath{\mathcal U}}
\newcommand{\cv}{\ensuremath{\mathcal V}}

\newcommand{\oo}{\ensuremath{\omega}}
\newcommand{\OO}{\ensuremath{\Omega}}

\newcommand{\sig}{\ensuremath{\sigma}}

\newcommand{\sm}{\backslash}

\newcommand{\isom}{\cong}

\newcommand{\cls}[1]{\ensuremath{\overline{#1}}}
\newcommand{\act}{\curvearrowright}

\makeatletter
\DeclareRobustCommand{\cev}[1]{%
  \mathpalette\do@cev{#1}%
}
\newcommand{\do@cev}[2]{%
  \fix@cev{#1}{+}%
  \reflectbox{$\m@th#1\vec{\reflectbox{$\fix@cev{#1}{-}\m@th#1#2\fix@cev{#1}{+}$}}$}%
  \fix@cev{#1}{-}%
}
\newcommand{\fix@cev}[2]{%
  \ifx#1\displaystyle
    \mkern#23mu
  \else
    \ifx#1\textstyle
      \mkern#23mu
    \else
      \ifx#1\scriptstyle
        \mkern#22mu
      \else
        \mkern#22mu
      \fi
    \fi
  \fi
}

\makeatother

\newcommand{\pth}[2]{\ensuremath{#1}\text{--}\ensuremath{#2}~path}

\newcommand{\pths}[2]{\ensuremath{#1}\text{--}\ensuremath{#2}~paths}
\newcommand{\arc}[2]{\ensuremath{#1}\text{--}\ensuremath{#2}~arc}

\newcommand{\seq}[1]{\ensuremath{(#1_n)_{n\in\N}}} 

\newcommand{\g}{\ensuremath{G\ }}
\newcommand{\G}{\ensuremath{G}}

\newcommand{\Lr}[1]{Lemma~\ref{#1}}
\newcommand{\Lrs}[1]{Lemmas~\ref{#1}}
\newcommand{\Tr}[1]{Theorem~\ref{#1}}
\newcommand{\Trs}[1]{Theorems~\ref{#1}}
\newcommand{\Sr}[1]{Section~\ref{#1}}

\newcommand{\Prr}[1]{Pro\-position~\ref{#1}}

\newcommand{\Cr}[1]{Corollary~\ref{#1}}

\newcommand{\lf}{locally finite}

\newcommand{\fg}{finite graph}

\newcommand{\Cg}{Cayley graph}

\renewcommand{\iff}{if and only if}
\newcommand{\fe}{for every}
\newcommand{\Fe}{For every}

\newcommand{\st}{such that}

\newcommand{\ti}{there is}

\newcommand{\obda}{without loss of generality}

\newcommand{\wrt}{with respect to}

\newcommand{\FC}{Freudenthal compactification}

\newcommand{\labtequ}[2]{
 \begin{equation} \label{#1} 	\begin{minipage}[c]{0.9\textwidth}  #2 \end{minipage} \ignorespacesafterend \end{equation} }

\newcommand{\mymargin}[1]{
 \ifnum \Debug = 1
  \marginpar{%
    \begin{minipage}{\marginparwidth}\small%
      \begin{flushleft}%
        {\color{blue}#1}%
      \end{flushleft}%
   \end{minipage}%
  }%
 \fi
}%

\newcommand{\mySection}[2]{}

\usepackage{mathtools,enumitem}
\usepackage[percent]{overpic}

\usepackage{authblk}
\usepackage[greek,english]{babel}  
\usepackage{hyperref}
\RequirePackage{latexsym} \RequirePackage{amsthm}
\RequirePackage{amsmath}

\usepackage{enumerate}
\RequirePackage{amssymb} \RequirePackage{makeidx}
\usepackage{dsfont}
\usepackage{mathrsfs}

\newcommand{\pd}{properly discontinuous}

\newcommand{\PF}{\ensuremath{\mathcal P}}
\newcommand{\SF}{\ensuremath{\mathcal S}}
\newcommand{\RF}{\ensuremath{\mathcal R}}

\newcommand{\cova}{covariant}
\newcommand{\covaly}{\cova ly}
\renewcommand{\fg}{finitely generated}

\begin{document}

\title{On planar Cayley graphs and Kleinian groups}

\author{Agelos Georgakopoulos\thanks{Supported by the European Research Council (ERC) under the European Union's Horizon 2020 research and innovation programme (grant agreement No 639046).}
\medskip 
\\
  {Mathematics Institute}\\
 {University of Warwick}\\
  {CV4 7AL, UK}\\
}

\date{\today}
\maketitle

\begin{abstract}



Let $G$ be a finitely generated group acting faithfully and properly discontinuously by homeomorphisms on a planar surface $X \subseteq \BS^2$. We prove that $G$ admits such an action that is in addition co-compact, provided we can replace $X$ by another surface $Y \subseteq \BS^2$. 

We also prove that if a group $H$ has a finitely generated Cayley\\ (multi-)graph $C$ covariantly embeddable in $\BS^2$, 
then $C$ can be chosen so as to have no infinite path on the boundary of a face. 

The proofs of these facts are intertwined, and the classes of groups they define coincide. In the orientation-preserving case they are exactly the (isomorphism types of) finitely generated Kleinian function groups. We construct a finitely generated  planar Cayley graph whose group is not in this class.

In passing, we observe that the Freudenthal compactification of every planar surface is homeomorphic to the sphere.

\end{abstract}


\section{Introduction}
\comment{
The study of discrete groups of isometries, or M\"obius transformations, of $\R^2$ and $\mathbb H^2$ is classic. It is at the heart of Klein's Erlangen program \cite{Klein} of understanding geometries by studying groups acting on them. Poincar\'e's  contributions to the topic are also famous; quoting \cite{ChGhLeSci},\\ {\it ``... a kind of apotheosis of 19th century mathematics, a meeting place for group theory, complex analysis, elliptic and modular functions, differential equations, Riemann surfaces, hyperbolic geomentry and quadratic forms ...''.}\\ 
According to \cite{ChGhLeSci} again, these groups ``...\ are tools used constantly today in various areas of mathematics and even in physics ...''. For example, they play an important role in the study of 3-manifolds \cite{Maskit,ThuThr,Thurston}. 
The groups alluded to here are generally called \defi{Kleinian} groups, 
although the precise  meaning of the term differs slightly depending on the author and era.
Some survey material and many further references can be found e.g.\ in \cite{LyndonSchupp,Marden,Maskit,Ohshika,SerCra,ZVC}.
}
The study of discrete groups of isometries, or M\"obius transformations, of $\R^2$ and $\mathbb H^2$ is classic. It is at the heart of Klein's Erlangen program \cite{Klein} as well as some of Poincar\'e's  most famous work \cite{ChGhLeSci}. According to \cite{ChGhLeSci}, these groups ``...\ are tools used constantly today in various areas of mathematics and even in physics ...''. For example, they play an important role in the study of 3-manifolds \cite{Maskit,ThuThr,Thurston}. 
The groups alluded to here are generally called \defi{Kleinian} groups, 
although the precise  meaning of the term differs slightly depending on the author and era.
Some survey material and many further references can be found e.g.\ in \cite{LyndonSchupp,Marden,Maskit,Ohshika,SerCra}.

A well-studied sub-family of the Kleinian groups are the \defi{function groups}, defined by the existence of an invariant component in their domain of discontinuity, see \Sr{sec Klein} for definitions, and \cite{MarGeo,Marden,MasCla,Maskit} for some literature. In this paper we show that every \fg\ function group is isomorphic to a group acting faithfully, \pd ly, and co-compactly on a planar surface. 

Levinson \& Maskit  \cite{LevMasSpe} proved that these groups are exactly the ones admitting a \fg\  \Cg\ that embeds in $\BS^2$ with a fixed cyclic ordering of the labels of the edges around each vertex. Thus function groups form a subfamily (proper, as shown in this paper) of the \defi{planar groups}, i.e.\ the groups having planar \Cg s, which are studied in recent work by the author and others \cite{ArzChe,DroInf,DrSeSeCon, DunPla, cayley3, planarPresI, planarPresII,KroInf,mohTre,ThoTil,TucFin}.
We will prove that each \fg\ function group admits a planar \Cg\ having no infinite facial path, by possibly allowing loops and parallel edges. This was of interest to the author since \cite{cayley3}, which provides constructions of 3-connected planar \Cg s in which no face is bounded by a finite cycle, previously thought to be impossible. These results apply more generally to groups containing orientation-reversing elements.

\comment{
\begin{theorem}\label{thm pd}
A finitely generated group has a \cova\ plane, finitely generated \Cg, \iff\ it admits a faithful, \pd\  action by homeomorphisms on a 2-manifold contained in the sphere $\BS^2$.

Moreover, this 2-manifold can be assumed to be the sphere, the plane $\R^2$, the open annulus, or the Cantor sphere, depending on whether the group has 0, 1, 2, or infinitely many ends, respectively. In addition, the action can be assumed to be co-compact.
\end{theorem} 

In other words:
}

\medskip
An embedding $\sigma: G \to \BS^2$ of a \Cg\ $G$ of a group $\Gamma$ is \defi{\cova}, if the canonical action of $\Gamma$ on $G$ maps every facial path onto a facial path; see \Sr{sec defs} for more precise definitions. (This is equivalent to saying that the canonical action $\Gamma \act \sigma(G)$ extends into an action $\Gamma \act \BS^2$ by homeomorphisms, see \Cr{cor hom S}.)  

A \defi{planar surface} is a connected 2-manifold homeomorphic to an (open) subset of the sphere $\BS^2$. Our main result is
\begin{theorem}\label{tfae}
For a finitely generated group $\Gamma$, the following are equivalent:
\begin{enumerate}[label=(\Alph*)]
\item \label{T i} $\Gamma$ admits a faithful, \pd\  action by homeomorphisms on a planar surface;
\item \label{T ii} $\Gamma$ has a \Cg\ admitting a \cova\ embedding; 
\item \label{T iv} $\Gamma$ has a Cayley multi-graph (see \Sr{sec GA}) admitting a \cova\ embedding every facial path of which is finite; 
\item \label{T iii} $\Gamma$ admits a faithful, \pd, {\bf \emph{co-compact}}  action by homeomorphisms on  the sphere, the plane $\R^2$, the open annulus, or the Cantor sphere.
\end{enumerate}
\end{theorem} 

In the orientation preserving case, the groups of \Tr{tfae} coincide, as abstract groups, with the Kleinian function groups mentioned above:
\begin{corollary}\label{Kleinian}
A finitely generated group $\Gamma$ admits a faithful, \pd\ (co-compact) action by orientation-preserving homeomorphisms on a planar surface \iff\ it is isomorphic to a Kleinian function group. 
\end{corollary} 
\Cr{Kleinian} can be deduced from \cite[THEOREM 4]{LevMasSpe}, which essentially says that the groups of \ref{T ii} coincide in the orientation-preserving case with the Kleinian function groups. We will give an alternative proof bypassing Ahlfors' finiteness theorem. 

Our proof of \Cr{Kleinian} makes use of a classical theorem of Maskit, saying that if $p:\tilde{S} \to S$ is a regular covering of a topologically finite surface $S$, where $\tilde{S} $ is planar, then the group of deck transformations is isomorphic to a function group. By exploiting the equivalence of \ref{T i} and  \ref{T iii} of \Tr{tfae} we can strengthen this statement by dropping the topological finiteness condition (\Cr{cor Maskit}); I do not know if this result is knew. (Maskit's formulation however is stronger than the one above, and it is not possible to strengthen it this way, see \Sr{proof} for details.)


In \Cr{Kleinian} we cannot replace `Kleinian function group' by just `Kleinian group': \Sr{sec rels} provides an explicit example  of a finitely generated Kleinian group that does not admit a planar \Cg\ (in fact most Kleinian groups have this property).

A description of the isomorphism types of Kleinian function groups in terms of fundamental groups of graphs of groups with simpler building blocks can be found in \cite{LevMasSpe}. Dunwoody \cite{DunPla} extends this to groups with planar \Cg s. 

\medskip
I suspect that \Cr{Kleinian} extends to the orientation reversing case, by using orientation reversing Kleinian groups, i.e.\  discrete subgroups of $PSL(2, \C).2$.
In support of this conjecture, we will also prove the following (in \Sr{sec R3}): 
\begin{theorem}\label{act R3}
Every group $\Gamma$ as in \Tr{tfae} admits a faithful \pd\ action on $\R^3$.
\end{theorem} 

We show, in \Sr{sec rels}, that the family of such groups is 
a proper subfamily of the groups admitting a planar \Cg.

I suspect that the restriction of $\Gamma$ being finitely generated can be dropped in all the above, except that in \ref{T iii} we would have to extend the list of possible surfaces; see the remark in \Sr{sketch}. Faithfulness is however necessary e.g.\ for the implication \ref{T i} $\to$ \ref{T ii}.

The groups of \Tr{tfae} can be described by a certain kind of group presentation, which allows them to be effectively enumerated \cite{planarPresI}. Thus we also obtain an effective enumeration of the isomorphism types of Kleinian function  groups.

\medskip
I do not know a proof of the implication \ref{T ii} $\to$ \ref{T iv} of \Tr{tfae} that does not go via \ref{T iii}.
I also do not know a proof of the implication \ref{T i} $\to$ \ref{T iii} that does not go via \ref{T ii}. However, if we relax \ref{T iii} by not requiring co-compactness, then it follows from the following more general and perhaps well-known  statement (\Sr{X}):
\begin{proposition}\label{ext pd i}
Let $\Gamma \act X$ be a \pd\ action on a metrizable, arc-connected, locally compact space $X$. Then the canonical extension $\Gamma \act (X \cup \OO^V(X))$ of the action to the non-accumulation ends $\OO^V(X)$ of $X$ is \pd.
\end{proposition}

We remark that although any action $\Gamma \act X$ as in \Tr{tfae} \ref{T iii} defines a quotient orbifold $O:= X/\Gamma$, understanding these groups goes beyond understanding 2-orbifold fundamental groups, because $X$ is not simply connected in the cases we are most interested in, see \Sr{sec orb}.

\medskip
Any action $\Gamma\act X$ as in \ref{T iii} can be `geometrised', to turn it into an action by isometries on a smooth manifold homeomorphic to one of those four spaces, see \Sr{proof}.

\medskip 
In passing, we observe the following purely topological statement (\Sr{sec back})
\begin{corollary} \label{cor FX i} 
Let $X \subseteq \BS^2$ be a 2-manifold. Then the \FC\ of $X$ is homeomorphic to $\BS^2$.
\end{corollary}

Part of the motivation behind this paper was to understand which planar surfaces admit an action by an infinite group. \Prr{ext pd i}, and the discussion in \Sr{X} sheds some light into this question. One way to produce such surfaces is to start from an action as in \Tr{tfae} \ref{T iii}, remove any totally disconnected subspace of the quotient space, and lift those punctures back to the original space. But this does not account for all such surfaces, see the Remark after \Cr{cor four}. I think that it is possible to give a full list of those surfaces by pursuing these ideas further, but this will be rather tedious and will not add much to our understanding. Much more interesting would be to extend \Tr{tfae} \ref{T iii} to higher dimensions:

\begin{problem}
Is there for every $n\geq 3$ a finite list $M_n$ of $n$-manifolds \st\ if a (finitely generated) group $\Gamma$ acts faithfully and \pd ly on any $n$-manifold contained in $\R^n$, then $\Gamma$ acts faithfully, \pd ly and co-compactly on an element of $M_n$?
\end{problem}

\Tr{act R3} motivates
\begin{problem}
Let $\Gamma$ be a group acting faithfully and \pd ly on an $n$-manifold contained in $\R^n$. Does $\Gamma$ act faithfully and \pd ly on $\R^{n+1}$? On $\R^{n+f(n)}$ for some function $f: \N \to \N$?
\end{problem}

In this paper we proved the case $n=2$, and this needs all the implications \ref{T i} $\to$ \ref{T ii} $\to$ \ref{T iii} $\to$ \ref{T iv} $\to$ \Tr{act R3}, so it is not straightforward to adapt our proof to higher dimensions. 

\medskip
The proof of \Trs{tfae} and~\ref{Kleinian} span Sections~\ref{sec cocomp}--\ref{proof}. \Tr{act R3} needs the full strength of \Tr{tfae} as it is based on \ref{T iv}. We prove it in \Sr{sec R3}. 

\subsection{Proof ideas} \label{sketch}
We now sketch the main ideas behind the proof of \Tr{tfae}. 

A central notion in the study of groups acting on surfaces, going back at least as far as Poincare's polyhedron theorem \cite{Maskit}, is that of a \defi{fundamental domain}. For the purposes of this sketch, let us say that a fundamental domain of the action $\Gamma \act X$ is a subset  $D\subset X$ containing exactly one point from each orbit of $\Gamma \act X$. Thus $\Gamma$ is in bijection with the set of translates of $D$, and the nicer $D$ is, the easier it makes it to understand the action. In general, one wants $D$ to be connected, and the ideal situation is when the closure of $D$ is a polygon, with its translates giving a locally finite tessellation of $X$. This way one obtains e.g.\ regular tessellations of the hyperbolic and euclidean plane when $\Gamma$ is a crystallographic group. 

When $\Gamma$ acts by arbitrary homeomorphisms rather than isometries however, and especially when $X$ is the Cantor sphere, then it is not easy to find useful fundamental domains. Instead, we will work with fundamental domains of \emph{graphs} embedded in $X$ upon which $\Gamma$ acts: a \defi{fundamental domain} in this sense will be a connected subgraph containing exactly one vertex from each $\Gamma$-orbit. We will make extensive use of an observation of Babai \cite{BabCon} (\Sr{sec Bab}), that if $\Gamma$ acts freely on a connected graph $H$, then contracting each translate of a fundamental domain into a vertex turns $H$ into a \Cg\ of $\Gamma$. 

Let me explain how this helps to prove  \Tr{tfae}.
Suppose $\Gamma$ is a finite group, and $\Gamma \act  \BS^2$ a faithful action by homeomorphisms.
As a warm-up exercise, let me sketch a proof that $\Gamma$ has a planar \Cg, which is the easiest special case of the implication \ref{T i} $\to$ \ref{T ii}. Easily, there is a point $p\in \BS^2$ with trivial stabiliser in $\Gamma \act  \BS^2$. In other words, the orbit $V$ of $p$ is in bijection with $\Gamma$. Let  $G$ be any \Cg\ of $\Gamma$, for example the complete graph on $\Gamma$. We identify the vertex set of $G$ with $V$ using the aforementioned bijection, and we map each edge of $G$ to an arc in $\BS^2$ between its end-vertices \st\ the action $\Gamma \act  \BS^2$ permutes these arcs. This can easily be achieved by choosing the arcs corresponding to edges of $G$ going out of a reference vertex, and defining the rest of the arcs as the images of the former under $\Gamma \act  \BS^2$. The map of $G$ into $\BS^2$ thus defined is not necessarily an embedding, as these arcs may intersect each other. However, it is easy to choose the arcs so that they intersect in at most finitely many points. Treating these intersection points as vertices defines a new graph $G'$. This $G'$ is planar by definition. It is not a \Cg, although $\Gamma$ acts on it by restricting $\Gamma \act  \BS^2$ to $G'$. With a slight modification, only needed if $\Gamma$ contains involutions, we can ensure that the latter action $\Gamma \act G'$ is free. By Babai's aforementioned result, $G'$ contracts into a \Cg\ $C$ of $\Gamma$. Since the contracted sets are connected, and $G'$ was planar, so is its minor $C$. We have found a planar \Cg\ of $\Gamma$, proving in particular that it is one of Maschke's groups \cite{Maschke}.

The same technique works when $\Gamma$ is infinite but finitely generated and acts \pd ly on a proper sub-surface of $X\subset \BS^2$. Proper discontinuity is important for ensuring that each arc is intersected at most finitely often when mapping $G$ into $X$, but the rest of the proof is essentially the same, and it yields the implication \ref{T i} $\to$ \ref{T ii}. 

The requirement that $\Gamma$ be finitely generated was useful here in order to guarantee that each arc is intersected at most finitely often. If we drop it, and some arcs are intersected infinitely often, perhaps we can still control those intersections so as to get a graph-like space in the sense of \cite{ThomassenVellaContinua}.  I suspect that Babai's result can be extended to such spaces, and that this can be used to generalise our proofs to the infinitely generated case, but this will require additional work.

For \ref{T ii} $\to$ \ref{T iii} the main machinery is the work of Thomassen \& Richter \cite{ThomassenRichter}, showing that if $G$ is a 3-connected planar graph, then its \FC\ $|G|$ embeds in $\BS^2$, and this embedding is unique up to modifying it by a homeomorphism of $\BS^2$ (the corresponding statement for finite graphs is a classical result of Whitney), see \Sr{sec TR} for details. Given a \Cg\ $G$ as in \ref{T ii}, we extend $\Gamma \act G$ to the faces of an embedding of $|G|$ into $\BS^2$, and after removing the images of the ends of $G$ we are left with a \pd\ action on one of the four spaces in \ref{T iii}, because every finitely generated \Cg\ has either at most two or a Cantor set of ends. Co-compactness is a byproduct of this construction. Some technical difficulties arise from the fact that our graphs are not necessarily 3-connected, and are overcome by embedding them into 3-connected graphs on which $\Gamma$ acts freely but not transitively (\Lr{ext 3cd}).

To go from \ref{T iii} to \ref{T iv}, we revisit the above proof of \ref{T i} $\to$ \ref{T ii}. We are given a plane \Cg\ $G$ of $\Gamma$ some faces of which have infinite paths of $G$ in their boundaries. We extend $G$ by adding some further generators, and map the corresponding edges into arcs in $\BS^2$ that cut up all such faces into smaller faces bounded by cycles, when again we treat intersection points of these arcs as new vertices in an auxiliary graph $G'$. The fact that this is possible with only finitely many additional generators is not obvious: it requires Dunwoody's \cite{DunAcc} result that finitely presented groups are accessible, combined with Droms' \cite{DroInf} result that planar groups are finitely presented, see \Lr{cutspace}. We then apply  Babai's contraction result as above to $G'$ to obtain a \Cg\ of $\Gamma$, and show that the property that all faces are bounded by finite cycles is hereby preserved.
\medskip

For the proof of \Tr{act R3}, we use \ref{T iv} and the main result of \cite{planarPresI}. The latter states that every 
\covaly\ planar Cayley graph $G$ is the 1-skeleton of a Cayley complex $Z$, which complex can be mapped into $\BS^2$ in such a way that (a) the restriction of the map to $G$ is \cova, and (b) the images of any two 2-cells of $Z$ are \defi{nested}, i.e.\ either disjoint, or one contained in the other (\Lr{pla pres}). The latter property allows us to map the 2-cells of $Z$ injectively into the inside $B^3$ of $\BS^2$ in $\R^3$, so that the image of $Z$ separates $B^3$ into 3-dimensional `chambers'. Using \ref{T iv} and some refinements on \Lr{pla pres} proved in \Sr{sec R3}, we can control the boundary of those chambers, so that we can use them to extend the action $\Gamma \act Z$ into a \pd\ action on $B^3$. 


\comment{
\begin{problem}
A finitely generated group $\Gamma$ acts faithfully and \pd ly on $\R^{n}$ for some $n\in \N$ \iff\ $\Gamma$ is finitely presented.
\end{problem}
}

\section{Definitions} \label{sec defs}

\subsection{Graphs} \label{def graphs}

We follow the terminology of \cite{diestelBook05}.

A 1-way infinite path is
called a \defi{ray}, a 2-way infinite path is
a \defi{double ray}.  

Two rays $R,L$ in $G$ are \defi{equivalent} if no finite set of vertices
separates them; we denote this fact by $R\approx_G L$, or simply by $R\approx L$ if $G$ is fixed. The corresponding equivalence
classes of rays are the \defi{ends} of $G$. We
denote the set of these ends by $\Omega =
\Omega(G)$. 

\subsection{Embeddings in the plane} \label{secDem}

An {\em embedding} of a graph \g will always mean a topological embedding of the corresponding 1-complex in the sphere $\BS^2$; in simpler words, an embedding is a drawing in $\BS^2$ with no two edges crossing. 

More generally, an embedding of a topological space $X$ in a topological space $Y$, is a map $\sigma: X\to Y$ which is a homeomorphism of $X$ to its image $\sigma(X)$.

A \defi{plane} graph is a graph endowed with a fixed embedding. A graph is \defi{planar}, if it admits an embedding. 

A {\em face} of an embedding $\sig: X \to \BS^2$, where $X$ is a topological space, is a component of $\BS^2 \sm \sig(X)$. If $X$ is a graph, or the \FC\ of a graph (see \Sr{sec FC}), we will say that  the face $F$ has {\em finite boundary}, if $\partial F$ contains the images of only finitely many vertices and edges of \G. 
Note that in this case  $\partial F$ is a cycle of \G.

A walk or path in \g is called {\em facial} with respect to \sig\ if it is contained in the boundary of some face of \sig. 

\subsection{Group actions} \label{sec GA}

Given a group $\Gamma$ and a generating set $S \subset \Gamma$, we define the (right) \defi{\Cg} $G= Cay(\Gamma,S)$ to be the graph with vertex set $V(G)= \Gamma$ and edge set $E(G)= \{ g(gs) \mid g\in \Gamma, s\in S\}$. We consider $G$ as a directed, labelled graph, with the edge $g(gs)$ being directed from $g$ to $gs$ and labelled by the generator $s$. Here, we are assuming that $S$ generates $\Gamma$, so that all \Cg s in this paper are connected. The group $\Gamma$ acts on $G$ by automorphisms, by multiplication on the left. 
By allowing $S$ to be a multi-set, possibly containing the identity, in the above definition we obtain a \defi{Cayley multi-graph}, with parallel edges and loops. For most of the paper parallel edges and loops, but they can matter in \ref{T iv} of \Tr{tfae}.

The \emph{Cayley complex} $C_\PF$ corresponding to a group presentation $\PF=\left< \SF \mid \mathcal \RF \right>$ is the 2-complex obtained from the \Cg\ $G$ of $\PF$ by glueing a 2-cell along each closed walk of $G$ induced by a relator $R\in \RF$. Here, a \defi{walk} $W$ is a sequence $v_0, v_1, \ldots v_k$ of vertices, \st\ each $v_i$ is joined to $v_{i+1}$ with an edge of \G; it is \defi{closed} when $v_k=v_1$. We say that $W$ is \defi{induced} by a relator $R$, if $R$ has exactly $k$ letters and the label of the edge from $v_i$ to $v_{i+1}$  coincides modulo $k$ with the $(i+n)$th letter of $R$ \fe\ $i$ and some fixed $n$.  

\medskip
Given a topological space $X$ and a group $\Gamma$ acting on it, the images of a  point $x\in X$ under the action of $\Gamma$ form the \defi{orbit} of $x$. A \defi{fundamental domain} is a subset of $X$ which contains exactly one point from each of these orbits.

An action $\Gamma \act X$ is \defi{\pd}, if it satisfies any of the following equivalent conditions:
\begin{enumerate} 
\item \label{pd i} \fe\ compact subspace $K$ of $X$, the set $\{g\in\Gamma \mid g K \cap K \neq \emptyset\}$ is finite;
\item \label{pd ii} \fe\ two compact subspaces $K,K'$ of $X$, the set $\{g\in\Gamma \mid g K \cap K' \neq \emptyset\}$ is finite;
\item \label{pd iii}  \fe\ $x,y\in X$ there are open neighbourhoods $U_x \ni x, U_y \ni y$ \st\ $U_x$ intersects $gU_y$ for at most finitely many $g\in \Gamma$.
\end{enumerate}
To see the equivalence of \ref{pd i} and \ref{pd ii} it suffices to notice that $K\cup K'$ is compact. 
 A proof of the equivalence of \ref{pd i} and \ref{pd iii} can be found in \cite{KapPD}, which offers many more equivalent definitions.

An action $\Gamma \act X$ is \defi{faithful}, if for every two distinct $g,h \in G$ there exists an $x \in X$ such that $gx \neq hx$; or equivalently, if for each $g \neq e \in G$ there exists an $x \in X$ such that $gx \neq x$. It is \defi{free} if $gx \neq hx$ \fe\ $g,h\in \Gamma$ and $x\in X$. It is \defi{transitive} if \fe\ $x,y\in X$ \ti\ $g\in \Gamma$ with $gx =y$ (we will only encounter transitive actions on discrete spaces $X$). Finally, $\Gamma \act X$ is \defi{regular} if it is free and transient.

An action $\Gamma \act X$ is \defi{co-compact}, if the quotient space $X/\Gamma$ is compact. If $X$ is locally compact, then an equivalent condition is that there is a compact subset $K$ of $X$ such that $\bigcup \Gamma K =X$. 

\subsection{Covariant embeddings}

Let \g be a graph and $\sigma: G \to \BS^2$ an embedding into the sphere. We say that $\sigma$ is \defi{\cova} \wrt\ a group action $\Gamma \act G$, if every element of $\Gamma$ maps each facial path of $\sigma$ into a facial path. To simplify notation, if $G$ is a \Cg\ of $\Gamma$ then we just say that $\sigma$ is {\cova} in this case. If \g is a plane graph, then we say that \g  is  \defi{\cova} if its embedding is \cova. If \g admits a \cova\ embedding then we say that it is \defi{\covaly\ planar}.

\subsection{Kleinian groups} \label{sec Klein}
A \defi{Kleinian group} is a discrete subgroup of $PSL(2, \C)$. Since $PSL(2, \C)$ is isomorphic to the group of M\"obius transformations of the Riemann sphere, every Kleinian group $\Gamma$ comes with a canonical action $\Gamma \act \BS^2$. The set of accumulation points of orbits of $\Gamma \act \BS^2$ is the \defi{limit set} $\Lambda(\Gamma)$, and the \defi{domain of discontinuity} is $\BS^2 \sm \Lambda(\Gamma)$. It is easy to check that $\Gamma$ acts \pd\ ly on the latter.

\subsection{Topology}

The \defi{boundary} $\partial U$ of a subset $U$ of a topological space $X$ comprises the points $x\in X$ \st\ every open neighbourhood of $x$ intersects both $U$ and $X\sm U$. The closure $\cls{U}$ of $U$ is the set $U \cup \partial U$.

The \defi{Cantor sphere} is the topological space $\cc$ obtained by removing from $\BS^2$ any subspace $S$ homeomorphic to the Cantor set. The well-known fact that $\cc$ does not depend on the particular choice of $S$ follows from Richards' classification of noncompact surfaces \cite{Richards}. The following well-known fact provides some explanation why the Cantor set is important for us.

\begin{proposition}[\cite{BroPer}] \label{Cantor}
Every nonempty totally disconnected perfect compact metrizable space is homeomorphic to the Cantor set
\end{proposition} 

\comment{
	The following is a consequence of Caratheodory's theorem, see \cite{DouHubEtu} or \cite{ThomassenRichter}[Lemma~4].
\begin{lemma} \label{Carath}	
Let $K$ be a compact, connected, locally connected subset of $\BS^2$, and let $U$ be a face of $K$. Then $\partial U$ is connected and locally connected. 
	\end{lemma}
}

A topological space is \defi{$k$-connected} for some $k\in \N$, if it is connected and remains so upon the deletion of any $k-1$ points.

A topological space is \defi{arc-connected}, if it contains a homeomorph of the real unit interval joining any two of its points.

\medskip
An $n$-\defi{manifold} is a metric space $X$ such that each point $x\in X$ has an open neighbourhood homeomorphic to $\R^n$. Smooth manifold structures are irrelevant in this paper, except shortly in \Sr{proof}.

\subsection{The \FC} \label{sec FC}
Even more than the above, this subsection is to be used as a reminder and for fixing notation; readers unfamiliar with the \FC\ are advised to consult a textbook for Topology.
\smallskip

Let $X$ be a topological space, and suppose that
$K_1\subseteq K_2 \subseteq K_3 ⊆ \ldots$ 
is a  sequence of compact subsets of $X$ whose interiors cover $X$.

An \defi{end} of $X$ is an equivalence class of nested sequences $U_1 \supseteq U_2 \supseteq U_3 \ldots$, where each $U_i$ is a  (connected) component of $X \sm K_i$, and two such sequences \seq{U},\seq{V}\ are declared to be \defi{equivalent}, if each $U_i$ contains $V_j$ for sufficiently large $j$ and conversely, each $V_i$ contains $U_k$ for sufficiently large $k$.

A space $X$ admitting a sequence \seq{K}\ as above is called \defi{hemicompact}. Examples include all connected manifolds, all connected locally-finite graphs viewed as 1-complexes, and more generally, all connected locally-finite simplicial compexes.

The set of ends \defi{\OO(X)} of $X$ is used to define a compactification $|X|:= X \cup \OO(X)$, called the \defi{\FC} or the end comactification of $X$. A basis of open neighbourhoods for $|X|$ can be obtained from one for $X$ by declaring $U \cup \OO(U)$ to be a basic open set whenever $U$ is a component of $X \sm K_i$ for some $K_i$ as above, and $\OO(U)$ is the set of equivalence classes of sequences $U_1 \supseteq U_2 \supseteq U_3 \ldots$ where $U_i \subseteq U$ for large enough $i$.

It is straightforward to check that when $X$ is a connected, locally finite graph, then this definition of $\OO(X)$ coincides with the combinatorial one from \Sr{def graphs}.

\section{Preliminaries}

\subsection{The extensions of Whitney's theorem by Thomassen \& Richter} \label{sec TR}

\begin{lemma}[{\cite[Proposition~3]{ThomassenRichter}}] \label{scc}	
Let $K$ be a compact, 2-connected, locally connected subset of $\BS^2$. Then every face of $K$ is bounded by a simple closed curve (contained in $K$, since $K$ is closed).
\end{lemma}

\begin{lemma}[{\cite[Lemma~12]{ThomassenRichter}}] \label{FG embeds}	
Let $G$ be a locally finite 2-connected planar graph.  Then the \FC\ $|G|$ of \g embeds in $\BS^2$.
\end{lemma}

The following classical result of Whitney \cite[Theorem~11]{whitney_congruent_1932} (which easily extends to infinite graphs by compactness, see e.g. \cite{ImWhi})
says that 3-connected planar graphs have an essentially unique embedding.

\begin{theorem}[Whitney's theorem] \label{imrcb}
Let \g be a 3-connected graph embedded in the sphere. Then every automorphism of \g maps each facial path to a facial path.
\end{theorem}

Thomassen \& Richter extended Whitney's theorem to the Freudenthal compactification of an infinite graph:
\begin{lemma}[{\cite[Theorem~2]{ThomassenRichter}}] \label{TR hom S}	
Let $G$ be a 3-connected planar graph. For every $g\in Aut(G)$, and every embedding $\sigma: |G| \to \BS^2$, there is a homeomorphism $h_g: \BS^2\to \BS^2$ \st\ $h_g \circ \sigma = g$. 

Even more, given any face $F$ of $\sigma$, and any homeomorphism $f: \cls{F} \to \cls{h_g(F)}$, we may assume that $h_g$ coincides with $f$ on $\cls{F}$.\footnote{The second statement of \Lr{TR hom S} is not stated explicitely in {\cite[Theorem~2]{ThomassenRichter}} but is implied by its proof.}
\end{lemma}

In fact, Thomassen \& Richter proved a more general version of \Lr{TR hom S}, which we will also use:
\begin{lemma}[{\cite[Theorem~2]{ThomassenRichter}}] \label{TR hom S gen}	
Let $X$ be a 3-connected, compact, locally connected Hausdorff space, admitting an embedding in $\BS^2$. Then $X$ is uniquely planar.
\end{lemma}

\subsection{Babai's contraction lemma} \label{sec Bab}

\begin{lemma}[{\cite{BabCon}}] \label{babai}
Let $\Gamma$ be a group acting freely on a 
connected graph \G. Then \ti\ a connected subgraph $D$ of $G$ meeting each $\Gamma$-orbit at exactly one vertex, \st\ the contraction $G/D$ is a Cayley graph of $\Gamma$.
\end{lemma}
Here, $G/D$ is the graph obtained from $G$ by contracting each $\Gamma$-translate of $D$ into a vertex. One can think of $D$ as the graph-theoretic analogue of a fundamental domain. 

\subsection{Existence of a regular orbit}
Given a group action $\Gamma \act X$, the \defi{orbit} $\Gamma x$ of a point $x\in X$ is the set $\{gx \mid g\in \Gamma\}$. 

The following is rather an exercise: 
\begin{lemma} \label{reg orb}
For every faithful, \pd\ action $\Gamma \act X$ on a connected mani\-fold $X$, \ti\ an orbit $O=\Gamma x$ \st\ the restriction of the action to $O$ is \defi{regular}. 
\end{lemma}
One way to prove this for example is using the well-known facts that  
the quotient space of every \pd\ action on a manifold is an orbifold (see e.g.\ \cite[Proposition~20]{BorRie}), and that the singular locus of any orbifold has empty interior (see e.g.\ \cite[Proposition~26]{BorRie}); in particular, \ti\ at least one point of the quotient that does not lie in the singular locus, which means by definition that its preimages have trivial stabilisers.

\subsection{Richards' classification of non-compact 2-manifolds}
The following is the special case of Richards' \cite{Richards}\footnote{Richards calls this `Ker\'ekj\'art\'o's theorem', but mentions that `Ker\'ekj\'art\'o's proof seems to contain certain gaps'.} classification of non-compact 2-manifolds when restricting to subspaces of $\BS^2$.
\begin{theorem}[\cite{Richards}] \label{Richards}
Let $X$ and $X'$ be two 2-manifolds contained in $\BS^2$. Then $X$ and $X'$ are homeomorphic \iff\ their Freudenthal boundaries are homeomorphic.
\end{theorem} 

\section{Separating faces and ends of planar graphs with cycles}

In this section we prove some basic facts about planar graphs that we will need later, which have nothing to do with group actions. 

If $\{V_1,V_2\}$ is a bipartition of the vertex set $V(G)$ of a graph \G, then the set of edges with exactly one endvertex in each $V_i$ is called a  \defi{cut} of \G.

The following is a rather trivial consequence of the definition of the \FC
\begin{lemma}[{\cite[Lemma~8.5.5~(ii)]{diestelBook05}}] \label{cut arc}
Let \g be a connected and locally finite graph, and $B$ a finite cut of \G. Then for every arc $A$ in $|G| \sm B$, the endpoints of $A$ lie in the same component of $G \sm B$.
\end{lemma}

The next two lemmas allow us to separate faces and ends of a plane graph by cycles.
\begin{lemma}\label{face sep}	
Let $G$ be a 2-connected graph and $\phi:|G| \to \BS^2$ an embedding. For any two faces $F,H$ of $\phi$, there is a cycle $C$ in $G$ \st\ $\phi(C)$ separates $F$ from $H$. 
Moreover, for every end $\oo\in \OO(G)$ \st\  $\phi(\oo)\not\in \partial F$, there is a  cycle $D$ in $G$ \st\ $\phi(D)$ separates $F$ from $\phi(\oo)$. 
\end{lemma}
\begin{proof}
If any of $F,H$ is bounded by a cycle $C$ of $G$ then we are done, so assume this is not the case. Let $P$ be a (possibly trivial) \pth{\partial F}{\partial H} in \G, and let $p=P\cap \partial F$ be its starting vertex. Let $u,v$ be the two neighbours of $p$ on $\partial F$, which are distinct because we are assuming $F$ is not bounded by two parallel edges. By \Lr{scc}, there is a $u$-$v$ arc in $|G|\sm P$, namely $\partial F$ with $p$ and its two incident edges removed. Thus, easily, there is also a \pth{u}{v} $R$ in $G - P$. We claim that the cycle $C:=Rvpu$ separates $F$ from $H$. Indeed, the path $P$, concatenated with any two arcs inside $F$ and $H$ respectively connects $F$ to $H$ and crosses $C$ exactly once (at $p$), and so $F,H$ lie in distinct sides of $\BS^2 \sm \phi(C)$.
\medskip

To separate $F$ from $\oo$, we follow the same idea, replacing $H$ with a basic open neighbourhood $O\ni \oo$, \st\ the closures $\cls{F},\cls{O}$ are disjoint. Let $P$ be a \pth{\partial F}{\partial O} in \G, and notice that $P\cap O=\emptyset$. As above, let $u,v$ be the two neighbours of $p:= P \cap \partial F$ on $\partial F$. There is now a  \pth{u}{v} $R'$ in $G \sm (P \cup \partial O)$ because $\partial F$ contains a $u$-$v$ arc in $|G|\sm (P \cup \partial O)$. Then $D:=R'vpu$ is a cycle separating $F$ from $\phi(\oo)$, because appending a ray of \oo\ inside $O$ to $P$ we can obtain an $F$-$\phi(\oo)$~arc crossing $D$ exactly once.
\end{proof}

\begin{lemma}\label{end sep}	
Let $G$ be a connected locally finite graph and $\phi:|G| \to \BS^2$ an embedding with no infinite face-boundaries. For any two ends $\oo,\chi\in \OO(G)$, there is a cycle $C$ in $G$ \st\  $\phi(\oo)$ and  $\phi(\chi)$ lie in distinct components of  $\BS^2 \sm \phi(C)$. 
\end{lemma}
\begin{proof}
Let $B$ be a finite cut of \g separating $\oo$ from $\chi$ in $|G|$, which exists by the definitions.

Let $U:= \bigcup \{ \partial F \mid F \text{ is a face of $\phi$ with } \partial F \cap B \neq \emptyset. \}$. Note that $B \subseteq U$. Let $H$ be the (finite) subgraph of \g induced by the edges in $U$. Note that $H$ is connected. We claim that $H \sm B$ is disconnected. Indeed, $H$ meets both sides of the cut $B$ of \G, and any path in $H$ between those sides would be a path in \g between the sides of $B$.

Next, we claim that $\phi(\oo)$ and $\phi(\chi)$ lie in distinct faces of $H$, where we think of $H$ as a plane graph with embedding $\phi(H)$.  To see this, let $A$ be an \arc{\oo}{\chi} in $|G|$. Easily, $A$ crosses $B$ an odd number of times because $B$ separates $\oo$ from $\chi$ (see \Lr{cut arc}). Since $B$ separates $H$, the endpoints of $\phi(A)$ lie in faces $F_\oo,F_\chi$ of $H$ separated by $B$, and so $F_\oo \neq F_\chi$. 

Then $\partial F_\oo$ contains a cycle $C$ \st\ $\phi(C)$ separates $F_\oo$ from $F_\chi$, and in particular $\phi(\oo)$ from $\phi(\chi)$ as desired.
\end{proof}

For a graph $K$ we let $E(K)$ denote its edge-set. The following says that if no element of a set of cycles in a plane graph separates two given points, then neither does the  sum (in simplicial homology) of those cycles.

\begin{lemma}\label{sep sum}	
Let $G\subset \BS^2$ be a plane graph and $\oo,\chi\in \BS^2 \sm G$. Let $K$ be a cycle of \g such that $E(K)= \sum_{1\leq i \leq k} E(C_i)$ where $C_i$ is a cycle of $G$ which does not separate $\oo$ from $\chi$, and the summation takes place in the cycle space of $G$. Then $K$ does not separate $\oo$ from $\chi$.
\end{lemma}
\begin{proof}
Choose an \arc{\oo}{\chi} $A\subset \BS^2$ that meets each edge of \g in at most one point. Note that any Jordan curve in $\BS^2$, and in particular any cycle of \G, separates $\oo$ from $\chi$ \iff\ it crosses $A$ an odd number of times. Thus each $C_i$ crosses $A$ an even number of times, hence so does $K$ since adding edge sets of cycles preserves the parity of the number of crossings of $A$. We conclude that $K$ does not separate $\oo$ from $\chi$.
\end{proof}

\section{From planar \Cg s to actions on 2-manifolds} \label{sec cocomp}

In this section we prove the implication \ref{T ii} $\to$ \ref{T iii} of \Tr{tfae}. This is mainly done in \Lr{pd ext}. 
We first collect a couple of lemmas.

\begin{lemma}\label{acc clos}	
Let $X$ be a connected, locally finite graph, and suppose $\Delta \subset Aut(X)$ fixes a non-empty set $Y\subset V(X)$ and the stabiliser of each  vertex of $X$ in $\Delta$ is finite. Then all accumulation points of any orbit of $\Delta \curvearrowright |X|$ lie in the closure \cls{Y} of $Y$ in the \FC\  $|X|$ of $X$.
\end{lemma}
Here, we say that $\Delta$ fixes $Y$ if $\Delta Y=Y$.
\begin{proof}
For $v\in Y$ it is clear that $\partial (\Delta v) \subseteq \cls{Y}$. For any other $w\in V(X)$, pick a \pth{w}{v} $P$ in $X$ with $v\in Y$. Suppose, for a contradiction, that $x\in \partial (\Delta w) \sm \cls{Y}$, and let $O\ni x$ be a basic open set disjoint from \cls{Y}. Since vertex stabilisers are finite, $gP$ meets the finite set $\partial O$ for only finitely many $g\in \Delta$. Thus almost all of $\{gP \mid g\in \Delta\}$ lie in $O$. Since this holds for an arbitrarily small neighbourhood $O$ of $x$, we deduce $x\in \partial (\Delta v)$, and so $x\in \cls{Y}$ by our first remark. This contradiction proves that $\partial (\Delta w) \subseteq \cls{Y}$ \fe\ $w\in V(X)$. 

It remains to show that $\partial (\Delta \oo) \subseteq \cls{Y}$ \fe\ end $\oo\in \OO(X)$. For this, let $z\in \partial (\Delta \oo)$, and suppose to the contrary that $z\not\in \cls{Y}$. Let $Z\ni z$ be a basic open set with $Z\cap \cls{Y} = \emptyset$. Let $R$ be a ray of $\oo$. We may assume that $R\subseteq Z$, for otherwise we can replace $\oo$ by another element $\oo'$ in its orbit $\Delta \oo$ for which this is true. Now for every vertex $r$ of $R$, we have $\partial (\Delta r) \not\in Z$ because we proved $\partial (\Delta r) \subseteq \cls{Y}$ above. Thus we can find a sequence $\seq{g}$ of elements of $\Delta$ with $g_n r \not\in Z$ \st\ $g_n R$ meets $Z$. Since $g_n R$ is connected, we can find a sequence $\seq{r}$ of vertices of $R$ \st\ $g_n r_n\in \partial Z$. As $\partial Z$ is finite, we may even achieve $g_n r_n = w\in \partial Z$ \fe\ $n$. But then $g_n^{-1} w \in V(R) \subseteq Z$, contradicting the fact that $\partial (\Delta w) \subseteq \cls{Y}$ \fe\ $w\in V(X)$.
\end{proof}

The following is a general tool for proving that a group action on a topological space is \pd.

\begin{lemma}\label{cover pd}
Let $X$ be a metrizable topological space, $\Gamma \act X$ an action by homeomorphisms, and $\mathcal U$ an open cover of $X$ \st\ \fe\ $U\in \mathcal U$, the orbit $\Gamma U$ has no accumulation point in $X$. Then the action $\Gamma \act X$ is \pd.
\end{lemma}
\begin{proof}
By the definitions, $\Gamma \act X$ is \pd\ unless \ti\ a compact subspace $K$ of $X$ and an infinite sequence \seq{g} of elements of $\Gamma$ \st\ $g_n K \cap K \neq \emptyset$ \fe\ $n\in\N$. 

In this case, let $\cu_K:= \{U\in \cu \mid U\cap K \neq \emptyset\}$ be the subset of $\cu$ intersecting $K$. Since $K$ is compact, $\cu_K$ has a finite subset $\cu_{K\infty}$ covering $K$. Since $g_n K \cap K \neq \emptyset$, we have $g_n U_n \cap K \neq \emptyset$ \fe\ $n\in\N$ and some $U_n\in \cu_{K\infty}$. As $\cu_{K\infty}$ is finite, we deduce that some $U_n\in \cu_{K\infty}$ has an accumulation point in $K$ because $K$ is metrizable, hence sequentially compact. This contradicts our assumptions.
\end{proof}

Given a planar graph \g and $\Gamma\subseteq Aut(G)$, we call an embedding $\sigma: G \to \BS^2$ \defi{$\Gamma$-\cova}, if every element of $\Gamma$ maps each face-boundary of $\sigma$ to a face-boundary of $\sigma$.

The following lemma was inspired by an idea of Dunwoody \cite{DunPla}.
\begin{lemma}\label{ext 3cd}
Let \g be a connected, locally finite graph, $\Gamma \subseteq Aut(G)$, and $\sigma: G \to \BS^2$ a $\Gamma$-\cova\ embedding. Then there is a 3-connected, locally finite supergraph $H$ of $G$ endowed with an extension $\Gamma \act H$ of $\Gamma \act G$ and a $\Gamma$-\cova\ embedding $\sigma': H \to \BS^2$ extending $\sigma$. 
Moreover, the maximum order of a vertex stabiliser of   $\Gamma \act H$ coincides with that of $\Gamma \act G$. Furthermore, the identity map on $V(G)$ gives rise to a canonical homeomorphism from $\OO(G)$ to $\OO(H)$.
\end{lemma} 
\begin{proof} 	
We simultaneously construct $H$ and its embedding $\sigma'$ as follows. For every facial double-ray $R$ of $\sigma$, we embed two copies $R', R''$ of $R$ in the face $F_R$ of $\sigma$ incident with $R$, we join each vertex of $R$ to the corresponding vertex of $R'$ with an arc, and we join each vertex of $R'$ to the corresponding vertex of $R''$ with an arc. Moreover, for every facial cycle $C$ of $\sigma$ that contains more than two edges, we embed a copy $C'$ of $C$ in the face of $\sigma$ incident with $R$, and join each vertex of $C$ to the corresponding vertex of $C'$ with an arc. Easily, we can embed all those copies and arcs so that the never meet each other except at common vertices. This defines the supergraph $H$ of $G$ and its embedding $\sigma'$. Since $\sigma$ was $\Gamma$-\cova, each element of $\Gamma$ extends canonically into an automorphism of $H$, still preserving face-boundaries, and the finiteness of vertex stabilisers. It is clear by the construction that $H$ has the same space of ends as $G$.

\medskip
It remains to check that $H$ is 3-connected. For this, suppose $\{x,y\} \subset V(H)$ disconnects $H$. Note that the subgraph of $H$ spanned by a triple of double-rays $R \cup R' \cup R''$ as in the above construction is 3-connected. Similarly, the subgraph of $H$ spanned by a couple $C \cup C'$ as above is 3-connected too. It follows that $\{x,y\} \subset V(G)$.

Let $w,z$ be vertices in distinct components of $H - \{x,y\}$. By the above argument, we may assume $w,z\in V(G)$. Let $P$ be a shortest \pth{w}{z} in $G$. Let $F_1 F_2 \ldots F_k$ be a sequence of faces of $\sigma$ \st\ $\partial F_1$ contains the first edge of $P$, and $\partial F_k$ contains the last edge of $P$, and \fe\ $i<k$, $F_i$ shares an edge $e_i$ with $F_{i+1}$ where $e_i$ is incident with $P$ but not contained in $P$. In other words, $F_1 \ldots F_k$ is a \pth{F_1}{F_k} in the dual $G^*$ that does not cross $P$; it can be obtained as the sequence of faces visited by an arc `parallel' to $\sigma(P)$ lying close enough to it. We will construct a \pth{w}{z} $Q$ in $H - \{x,y\}$, contradicting our assumption that $\{x,y\}$ disconnects $w$ from $z$.

For this, note that the each of the edges $e_i$ from above has exactly one endvertex $x_i$ in $P$, because $P$ was chosen to be a geodetic path. It is now easy to construct $Q$ as a concatenation of \pths{x_i}{x_{i-1}} of $H$ contained in $F_i$ and an appropriate initial and final path in $F_1$ and $F_k$, respectively, see  \fig{figQ}.
\end{proof}

   \begin{figure}[htbp]
   \centering
 \includegraphics[width=0.5\linewidth]{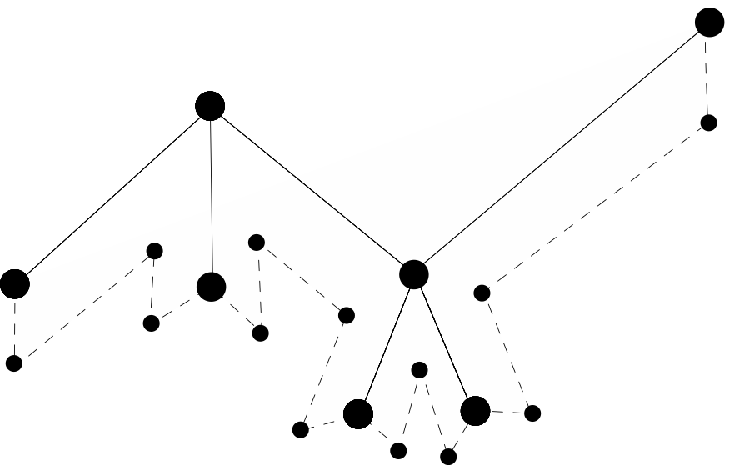}
 \put(-185,43){$w$}
\put(-76,55){$z$}
\put(-123,95){$x$}
\put(5,105){$y$}
\put(-123,95){$x$}
\put(-124,32){$x_1$}
\put(-131,64){$e_1$}
\put(-96,1){$x_2$}
\put(-58,3){$x_3$}
\put(-30,50){$Q$}

\caption{The path $Q$ in the proof of \Lr{ext 3cd} (dashed).} \label{figQ}
   \end{figure}

\comment{
\epsfxsize=0.55\hsize
\showFig{figQ}{The path $Q$ in the proof of \Lr{ext 3cd}.}

\begin{figure}
{\input{figQ.eps}}
\put(-179,43){$w$}
\put(184.5,86){$z$}
\put(-161,51){$x_i$}
\put(-132.5,76.5){$y$}
\put(-110,71.5){$E_i$}
\put(-86,76){$y'$}
\put(-61,50.5){$x_{i'}$}
\put(-33,55.5){$z'$}
\put(-31.5,42.5){$w'$}
\caption{The path $Q$ in the proof of \Lr{ext 3cd}.}
\label{figQ}
\end{figure}
}

Using \Lr{ext 3cd} we can immediately drop the condition of 2-connectedness in \Lr{FG embeds} for \covaly\ planar Cayley graphs:
\begin{corollary}\label{FG emb k1}
Let $G$ be a locally finite Cayley graph, and $\tau: G\to \BS^2$ a \cova\ embedding. Then \ti\ an embedding $\sigma: |G| \to \BS^2$ \st\ every walk of \g is facial in $\tau$ \iff\ it is facial in $\sigma$ (in particular, the restriction of $\sigma$ to $G$ is \cova).
\end{corollary}
\begin{proof}
We use  \Lr{ext 3cd} to embed \g into a 3-connected supergraph $H$. We then apply \Lr{FG embeds} to $H$, and restrict the resulting embedding from $|H|$ to its subspace $|G|$. Here, we used the fact that $\OO(H)$ is canonically homeomorphic to $\OO(G)$.
\end{proof}

It should be possible to extend this even further to every locally finite planar graph using the construction in the proof of \Lr{ext 3cd}, i.e.\ to just drop the 2-connectedness condition in \Lr{FG embeds}, but we will not need this. I suspect that  \Lr{FG embeds}  generalises even further as follows:
\begin{problem}
Let \g be a planar graph, and $\ell: E(G) \to \R_+$ an assignment of lengths to its edges. Let $\hat{G}$ denote the metric completion of the metric space defined by $\ell$ (see \cite{ltop} for the precise definition). Is $\hat{G}$ always homeomorphic to a subspace of $\BS^2$?
\end{problem}
 
Likewise, we can use \Lr{ext 3cd} to relax the 3-connectedness condition in \Lr{TR hom S}:
\begin{corollary} \label{cor hom S}	
Let $G$ be a connected, \lf\ graph, and $\Gamma \subseteq Aut(G)$. Let $\sigma: |G| \to \BS^2$ be a $\Gamma$-\cova\ embedding. Then \fe\ $g\in \Gamma$, there is a homeomorphism $h_g: \BS^2\to \BS^2$ \st\ $h_g \circ \sigma = g$. 

Even more, given any face $F$ of $\sigma$, and any homeomorphism $f: \cls{F} \to \cls{h_g(F)}$, we may assume that $h_g$ coincides with $f$ on $\cls{F}$.
\end{corollary}
\begin{proof}
Let $H$ be the 3-connected supergraph of \g and $\sigma': H \to \BS^2$ its\\ $\Gamma$-\cova\ embedding extending $\sigma(G)$, as provided by \Lr{ext 3cd}. As $\OO(H)$ is canonically homeomorphic to $\OO(G)$, we can in fact consider $\sigma'$ to be an embedding of $|H|$. Applying \Lr{TR hom S} to $H$ we obtain a homeomorphism $h_g: \BS^2\to \BS^2$ \st\ $h_g \circ \sigma' = g$. Hence $h_g \circ \sigma = g$ as $\sigma'$ extends $\sigma$.

For the second statement, given any face $F$ of $\sigma$, and a  homeomorphism $f: \cls{F} \to \cls{h_g(F)}$, we may assume that $f$ maps $F\cap \sigma'(H)$ homeomorphically onto  $h_g(F)\cap \sigma'(H)$ by constructing $\sigma'$ appropriately in the proof of \Lr{ext 3cd}. Applying the second statement of \Lr{TR hom S} to each face $F_i$ of $\sigma'$ with $F_i\subseteq F$ and the restriction of $f$ to $F_i$, we easily achieve that $h_g$ coincides with $f$ on $\cls{F}$.
\end{proof}

Let \g be a \Cg\ of a group $\Gamma$ with an embedding  $\sigma: G \to \BS^2$. We say that  $\sigma$ is \defi{orientation-preserving}, if the clockwise cyclic ordering in which the labels of the edges of a vertex $x$ of \g appear in $\sigma$ is independent of the choice of $x$. By \defi{labels} here we mean the corresponding element of the generating set of $\Gamma$ used to define $G$. (If \g is 3-connected, then Whitney's \Tr{imrcb} implies that this cyclic ordering is the same for each $x\in V(G)$ up to orientation. See e.g.\ \cite{cayley3} for more.)

\begin{lemma}\label{pd ext}
Let \g be a connected, locally finite graph, and suppose $\Gamma\subseteq Aut(G)$ has finite vertex stabilisers. Let $\sigma: |G| \to \BS^2$ be a $\Gamma$-\cova\ embedding. 
Then the action of $\Gamma$ on \g extends into  a (faithful) \pd\ action of $\Gamma$ on $\BS^2 \sm \sigma(\OO(G))$ by homeomorphisms. 

Moreover, if $\Gamma$ has finitely many orbits of vertices, then the latter action can be chosen to be co-compact.

Furthermore, if \g is a Cayley multi-graph of $\Gamma$, and $\sigma$ is orientation-preserving, and all faces of $\sigma$ have finite boundary, then the latter action can be chosen so that, in addition, it stabilises at most one point of each face of $\sigma$.
\end{lemma}
\begin{proof}
Given $\Gamma\subseteq Aut(G)$, we want to extend each $g\in \Gamma$ into a homeomorphism $h_g: \BS^2 \to \BS^2$. We obtain $h_g$ by applying \Cr{cor hom S}. In order for this map $g \mapsto h_g$ to extend the action $\Gamma \act G$ into an action $\Gamma \act \BS^2$, we need it to be a homomorphism from $\Gamma$ to $Aut(\BS^2)$. To achieve this, we will exploit the second statement of \Cr{cor hom S} appropriately. 

For this, given any map $\eta: \partial F \to \partial F$ which is the restriction of one or more $g\in \Gamma$ to the boundary $\partial F$ of some face $F$ of $\sigma$, we fix a homeomorphism $f_\eta: \cls{F} \to \cls{F}$ \st\ $f_\eta$ coincides with $\eta$ on $\partial F$. The existence of $f_\eta$ is a consequence of \Cr{cor hom S}, which unsurprisingly makes use of the Jordan-Sch\"onflies theorem. 

Similarly, if for two faces $F,F'$ of $\sigma$ there is $g\in \Gamma$ mapping $ \partial F' $ to $\partial F$, then we fix a map $\zeta=\zeta_{F',F}: \partial F' \to \partial F$ which is the restriction of such a $g$ to $\partial F'$, and a homeomorphism $f_\zeta: \cls{F}' \to \cls{F}$ \st\ $f_\zeta$ coincides with $\zeta$ on $\partial F'$.

By compositions of these $ f_\eta$'s and $f_\zeta$'s we obtain, for any $g\in \Gamma$ and any face $F$,  a unique map $f_F$ extending the restriction of $g$ to $\partial F$ to  $\cls{F}$.
When we apply \Cr{cor hom S} to define the map $g \mapsto h_g$ as above, we can, by its second statement, assume that the restriction of $h_g$ to each face $F$ coincides with this $f_F$. 

This immediately implies that $h_{gg'}= h_g h_{g'}$, hence our map $\phi: \Gamma \to Aut(\BS^2)$ defined by $g \mapsto h_g$ is an injective homomorphism. By identifying $\Gamma$ with its image in $Aut(\BS^2)$ under $\phi$ we obtain the action $\Gamma \act \BS^2$.

Let $X:= \BS^2 \sm \sigma(\OO(G))$. Since $\Gamma(\OO(G)) = \OO(G)$, we can define an action $\Gamma \act X$ by homeomorphisms just by restricting the above action from $\BS^2$ to $X$. 

\medskip
In the special case where all faces of $\sigma$ have finite boundary, 
this action turns out to be \pd\ and, if $V(G)$ has finitely many $\Gamma$-orbits, co-compact. We will first handle this special case as a warm-up towards the more involved general case.\footnote{In fact this special case can be handled by noting that $X$ is a CW-complex and $\Gamma\act X$ a cellular action, and using the equivalence between (2) and (10) in \cite[Theorem~9]{KapPD}. But we provide the following proof in order to ease the understanding of the proof of the general case.}

So let us first assume that all faces of $\sigma$ have finite boundary. 
To show that $\Gamma \act X$ is \pd\ we apply \Lr{cover pd} with the following choice of the cover \cu\ of $X$:
\begin{enumerate}
\item \label{cui} \fe\ $x\in \sigma(V(G))$, we choose an open neighbourhood $U_x$ of $x$ in $\BS^2$ meeting only edges of \g and faces of $\sigma$ incident with $x$, and containing the image of no other vertex of \G;
\item \label{cuii} \fe\ point $x\in \sigma(e)$ with $e\in E(G)$, we choose an open neighbourhood $U_x$ of $x$ contained in the union of $\sigma(e)$ with the faces of $\sigma$ incident with $e$;
\item \label{cuiii} \fe\ $x$ in a face $F$ of $\sigma$, we choose $U_x=F$;
\end{enumerate}
Let $\cu:= \{U_x \mid x\in X\}$ be the resulting cover of $X$. We claim that \cu\ satisfies the requirement of \Lr{cover pd} that $\Gamma U$ has no accumulation point in $X$ \fe\ $U\in \mathcal U$. To see this for $U=U_x$ of type \ref{cuiii}, note that $\Gamma$ maps every face of $\sigma$ to a face of $\sigma$, and that $U_x=F$ has a finite stabiliser in $\Gamma$ because it is determined by the finite set of vertices in $\partial F$ and we are assuming that every vertex has a finite stabiliser in $\Gamma$. For $U_x$ of type \ref{cui} or \ref{cuii} we remark that $U_x$ is contained in the closure of the union of finitely many $U$'s of type \ref{cuiii}, and repeat the same argument. 

To show that $\Gamma \act X$ is co-compact when $V(G)$ has finitely many orbits, pick $D\subset V(G)$ containing exactly one vertex from each orbit, and let $K$ be the union of the closures of the faces incident with vertices in $D$. Then $K$ is compact, and it is easy to see that $\bigcup \Gamma K= X$, which means that $\Gamma \act X$ is co-compact.

To prove the final statement, we modify $G$ into a triangulation $G'$ by adding a new vertex $v_F$ inside each face $F$ of $\sigma$, and joining $v_F$ to each vertex incident with $F$ with a new edge, which edge we draw as an arc inside $F$ to obtain an embedding $\sigma'$ of $G'$. 
Since $\sigma$ is $\Gamma$-\cova, $\Gamma \act G$ extends to an action on $G'$. Since $\Gamma$  stabilises no vertex of $G$, and it preserves the orientation of $\BS^2$, the only vertices of $G'$ it stabilises are the $v_F$'s. Repeating the above construction with $G'$ instead of $G$ yields the desired action, where the only points of $\BS^2 \sm \sigma(\OO(G))$ fixed are the images of the $v_F$'s. 

\medskip
We now consider the general case, where some faces of $\sigma$ may have infinite boundary. In this case the above argument fails because the stabiliser of such a face may be infinite. To circumvent this difficulty, we will subdivide such faces by embedding appropriate graphs inside them, thus embedding \g into a plane supergraph $G'$, in such a way that the faces of $G'$ have finite stabilisers in $\Gamma \act G'$ even if they have \defi{infinite boundary}\footnote{We say that a face $F$ has \defi{finite boundary}, if $\partial F$ contains finitely many edges. Otherwise we say that $F$ has \defi{infinite boundary}.}.

To simplify the exposition, we assume below that \g is 2-connected (in order to be able to apply \Lrs{scc} and \ref{face sep}). If not, then we apply the following arguments to the 3-connected supergraph $H$ of \g provided by \Lr{ext 3cd} to obtain the desired action for $H$, hence also for $G$ (the final sentence has already been proved in the above special case of finite face boundaries).

Our auxiliary supergraph $G'$ of $G$ is defined as follows. For every face $F$ of $\sigma$ (with infinite boundary), we `reflect' our embedding $\sigma: |G| \to \BS^2$ from the complement of $F$ into $F$; that is, we choose a homeomorphism $r: (\BS^2 \sm F) \to \cls{F}$ that coinsides with the identity on $\cls{F}$, which exists by the Jordan-Schoenflies theorem \cite{ThoJS} and the fact that $\partial F$ is a simple closed curve (\Lr{scc}). 
Then $r \circ \sigma$ embeds a copy of \g into $\cls{F}$. We call this copy of \g the \defi{shadow} of \g in $F$. 
We define $G'$ to be the plane graph obtained as the union of \g with all those shadows, one for each face of $\sigma$, and let $\sigma': G' \to \BS^2$ denote the corresponding embedding. In fact, we can extend  $\sigma'$ into an embedding of $|G'|$ in $\BS^2$.

It is not hard to see that the action $\Gamma \act G$ extends to an action $\Gamma \act G'$: as each $g\in \Gamma$ maps each face $F$ of $\sigma$ to a face $gF$ of $\sigma$ via the map $g \mapsto h_g$, we can let $g$ map the shadow of \g in $F$ to the shadow of \g in $gF$ via the unique automorphism of \g determined by the restriction of $g$ to $\partial F$. This defines the action $\Gamma \act G'$, and therefore an action $\Gamma \act |G'|$. Easily, $\sigma'$ is $\Gamma$-\cova\ too.

We claim that 
\labtequ{orbits}{every orbit of $\Gamma \act |G'|$ has all its accumulation points in $\OO(G) \subset \OO(G')$.}
Indeed, this follows from \Lr{acc clos}, applied with $X=G'$ and $Y=V(G)$.


Our next claim is that $G'$ is still 2-connected. Indeed, after removing any vertex $x\in V(G')$, the graph $G$ as well as any of its closed shadows is still connected since $G$ was assumed to be 2-connected. Moreover, any shadow $S$ is still connected to \g as $G$ and $S$ have several vertices in common by construction, and our claim easily follows.

Since 
$\sigma'$ is $\Gamma$-\cova, we can apply \Cr{cor hom S} to it. As we did in the special case where all faces of \g have finite boundary above, we exploit the second statement of \Cr{cor hom S}  (the maps $f_\eta$ now correspond to faces of $\sigma'$) so as to extend $\Gamma \act |G'|$ into an action $\Gamma \act \BS^2$ and, letting $X:= \BS^2 \sm \sigma(\OO(G))$ (and not $\BS^2 \sm \sigma(\OO(G'))$), into an action $\Gamma \act X$ by restriction.

As above, we use \Lr{cover pd} to show that $\Gamma \act X$ is \pd. The cover \cv\ of $X$ we use is similar to \cu\ from above; the only difference is that now $X$ contains points that are images of ends $\oo\in \OO(G')\sm \OO(G)$, for which we need to choose open sets in our cover. For every such end $\oo$, we choose a cycle $D_x$ in $G'$ separating $\oo$ from $\partial F$ where $F$ is the face of $\sigma$ containing $\sigma'(\oo)$, which cycle exists by the second part of \Lr{face sep}, and will be used in \ref{cviv} below. Our cover $\cv$ of $X$ is defined as follows:
\begin{enumerate}
\item \label{cvi} \fe\ $x\in \sigma'(V(G'))$, we choose an open neighbourhood $V_x$ of $x$ in $\BS^2$ meeting only edges of \g and faces of $\sigma'$ incident with $x$, and containing the image of no other vertex of $G'$;
\item \label{cvii} \fe\ point $x\in \sigma'(e)$ with $e\in E(G')$, we choose an open neighbourhood $V_x$ of $x$ contained in the union of $\sigma'(e)$ with the faces of $\sigma'$ incident with $e$;
\item \label{cviii} \fe\ $x$ in a face $F$ of $\sigma'$, we choose $V_x=F$;
\item \label{cviv} \fe\ $x\in\sigma'(\OO(G')\sm \OO(G))$, we let $V_x$ be the component of $\BS^2 \sm D_x$ containing $x$, where $D_x$ is the cycle defined above.
\end{enumerate}
Let $\cv:= \{V_x \mid x\in X\}$ be the resulting cover of $X$. 

Again, we claim that $\Gamma U$ has no accumulation point in $X$ \fe\ $U\in \mathcal U$. To prove this for $U=V_x$ of type \ref{cviii}, we recall that $\Gamma$ maps every face of $\sigma'$ to a face of $\sigma'$, and show that $V_x=F$ has a finite stabiliser in $\Gamma$. If $F$ coincides with an original face of $\sigma$ with finite boundary, then we showed this above. Otherwise, $F$ lies inside a face $H$ of $\sigma$ in which $\sigma'$ embedds a shadow of \G. In this case, applying \Lr{face sep} to that shadow yields a cycle $C_F$ separating $F$ from $\partial H$. By \eqref{orbits}, the orbit $\Gamma C_F$ of $C_F$ has all its accumulation points in $\OO(G)$ in the topology of $|G'|$. Since $\sigma'(|G'|)$ is a homeomorph of $|G'|$, this means that the orbit of the circle $\sigma'(C_F)$ under our action $\Gamma \act \BS^2$ has all its accumulation points in $\sigma'(\OO(G))$. Since $F$ is contained in one of the sides of $C_F$, this easily implies that the orbit of $V_x=F$ also has all its accumulation points in $\sigma'(\OO(G))$. Thus $V_x$ has no accumulation point in $X= \BS^2 \sm \sigma'(\OO(G))$ as desired.

For $U=V_x$ of type \ref{cviv} we use the same argument, replacing the cycle $C_F$ by $D_x$.

For $V_x$ of type \ref{cvi} or \ref{cvii} we argue as earlier: we remark that $V_x$ is contained in the closure of the union of finitely many $V$'s of type \ref{cviii}, and repeat the same argument. 

Thus  \cv\ satisfies the requirement of \Lr{cover pd}, and we deduce that $\Gamma \act X$ is \pd.

To show that $\Gamma \act X$ is co-compact when $V(G)$ has finitely many orbits, we argue as above, except that we now employ \Lr{face sep} as follows. Pick $D\subset V(G)$ containing exactly one vertex from each orbit. For every face $F$ of $G'$ incident with a vertex of $D$, let $C_F$ be a cycle of $G'$ separating $F$ from the complement of the  face $F'$ of $G$ containing $F$, which exists by \Lr{face sep} applied to the shadow of $G$ in $F'$. (If $F$ has finite boundary, we can just let $C_F=\partial F$.) Let $K$ be the union of the closures of the interiors of all these cycles $C_F$. Then $K$ is compact, and again we have $\bigcup \Gamma K= X$, which means that $\Gamma \act X$ is co-compact.

\end{proof}


\section{From actions on 2-manifolds to planar \Cg s} \label{sec back}

In this section we prove the implication \ref{T i} $\to$ \ref{T ii} of \Tr{tfae} (\Lr{lem vapf}). \smallskip

We say that a space $X \subseteq \BS^2$ is \defi{uniquely planar}, if \fe\ two embeddings $\sigma, \sigma': |X| \to \BS^2$ there is a homeomorphism $h: \BS^2\to \BS^2$ \st\ $h \circ \sigma = \sigma'$, and there is at least one such embedding $\sigma$. In other words, every automorphism of $|X|$ extends into a homeomorphism of $\BS^2$.

\begin{lemma}\label{Freu X}
Every connected 2-manifold $X \subseteq \BS^2$ is {uniquely planar}.
\end{lemma} 
\begin{proof} 	
Our first aim is to show that $|X|$ admits an embedding in $\BS^2$. We will do this by applying \Lr{FG embeds} on a triangulation of $X$.

Let $D$ be a locally finite triangulation of $X$, which exists by a well-known theorem of Rad\"o \cite{RadRiem,ZVC}, and let $T$ be its 1-skeleton. Then $T$ is a 2-connected graph by definition. Therefore, there is an embedding $\sigma: |T| \to \BS^2$ by \Lr{FG embeds}. In fact, we may assume $T$ to be 3-connected: we can subdivide each edge of $D$ by putting a new vertex at its midpoint, and triangulate each original triangle $\Delta$ of $D$ into 4 smaller triangles by adding the three edges joining the midpoints of the three edges of $\Delta$. If  $T$ is the 1-skeleton of the resulting triangulation, then it is indeed 3-connected because the neighbourhood of $\Delta$ is still connected after removing any 2 vertices from $\Delta$.

Assuming, as we now can, that $T$ is 3-connected, we deduce that our embedding $\sigma: |T| \to \BS^2$ is essentially unique when restricted to $T$ by (the infinite version of) Whitney's \Tr{imrcb}. Therefore, the 1-skeleton of each triangle of $D$ bounds a face of $\sigma$, because it does so in the embedding of $T$ induced by the identity on  $X \subseteq \BS^2$. Thus we can extend $\sigma$ from $T$ to $D$ by mapping each 2-cell into the corresponding face of $\sigma$. Since $T$ and $X$ have the same ends by the definitions, we have thus extended $\sigma$ into an embedding $\phi: |X| \to \BS^2$.

We can now apply \Lr{TR hom S gen} to deduce that $|X|$ is uniquely planar.
\end{proof}

As a consequence of the first claim of this proof, we obtain \Cr{cor FX i}, which we restate for convenience:
\begin{corollary} \label{cor FX}
Let $X \subseteq \BS^2$ be a 2-manifold. Then $|X|$ is homeomorphic to $\BS^2$.
\end{corollary}
\begin{proof} 	
We showed above that $|X|$ admits an embedding $\sigma$ in $\BS^2$. If $\sigma$ is not surjective, pick $x\in \BS^2 \sm \sigma(|X|)$, and let $F$ be the face of $\sigma$ containing $x$. Applying \Lr{scc} with $K:= \sigma(|X|)$, we obtain that $F$ is bounded by a simple closed curve $S \subset |X|$. Since $\OO(X)$ is totally disconnected, $\sigma^{-1}(S)$ contains a point $y$ of $X$. We obtain a contradiction as $y$ has a neighbourhood $U$ homeomorphic to $\R^2$, and every neighbourhood of $\sigma(y)$ meets $F$.
\end{proof}

We remark that \Cr{cor FX} does not extend to 3 dimensions: let $M$ be the inside of a torus embedded in $\R^3$. Clearly $M$ is a 3-manifold embeddable in $\BS^3$. But $|M|$ is not homeomorphic to $\BS^3$, and in fact it is not homeomorphic to a 3-manifold. Indeed, the boundary $|M|\sm M$ consists of a single point $x$, and $x$ has arbitrarily small open neighbourhoods $U$ \st\ $U\sm x$ is homeomorphic to $M$, hence not simply connected.\footnote{I thank Max Pitz for this observation.}

\medskip
We can now prove the main result of this section. Recall the definition of an orientation-preserving embedding $\sigma$ of a \Cg\ given before \Lr{pd ext}.
\begin{lemma}\label{lem vapf}
Let $\Gamma$ be a finitely generated group with a faithful and \pd\ action $\Gamma \act X$ on a 
connected 2-manifold $X$. Then there is a  finitely generated  \Cg\ \g of $\Gamma$ and an  embedding $\sigma: \g \to X$ \st\ $\sigma(V(G))$ has no accumulation points in $X$, 
and $\Gamma \act G$ coincides with the restriction of $\Gamma \act X$ to $\sigma(G)$.

Moreover, if $X \subseteq \BS^2$, then $\sigma$ defines a \cova\ embedding of $G$ into $\BS^2$. Furthermore, if $\Gamma \act X$ is orientation-preserving, then so is $\sigma$.
\end{lemma} 
\begin{proof} 	
Pick a finite generating set $S$ of $\Gamma$. Let $o$ be a point of $X$ \st\ the orbit $\Gamma o$ is regular, which exists by \Lr{reg orb}. \Fe\ $s\in S$, pick an \arc{o}{so} $R_s$ in $X$. By straightforward topological manipulations we may assume that these arcs and their translates do not intersect too much: we can assume that \fe\ $s,t\in S$, and any $g\in \Gamma$, the intersection $R_s \cap gR_t$ is either empty or just one point. This means that the union $\bigcup_{s\in S} \Gamma R_s$ of all these arcs and their translates defines a (plane) graph $H$ embedded in $X$, the vertex set of which consists of our regular orbit $\Gamma o$ and all the intersection points $R_s \cap gR_t$ of our arcs. (We could assume that no three arcs meet at a point, but we will not need to.) To show that $H$ is a graph, we need to check that no arc $R_s$ is intersected by infinitely many other arcs. This is the case because $\Gamma$ acts \pd ly, and the `star' $\bigcup_{s\in S} R_s$ of $o$ is compact. 
\mymargin{perhaps the only argument that fails if Gamma is not f.g.}

Since $\Gamma o$ is an orbit of $\Gamma$, and  $S$ generates $\Gamma$, it follows easily that $H$ is connected. Moreover, $\Gamma \act X$ defines an action $\Gamma \act H$ by restriction. We would like to apply Babai's contraction \Lr{babai} to contract $H$ onto a \Cg\ of $\Gamma$, but the freeness condition is not necessarily satisfied because some involution $g\in \Gamma$ might exchange two crossing arcs $R_s, gR_s$, hence stabilising their intersection point. But this is easy to amend by a slight modification of $H$: we blow up each vertex $v$ of $H$ arising from an intersection  point into a circle that intersects $H$ at the edges incident with $v$ only, intersecting each of them exactly one. Thus $v$ is replaced in $H$ by a cycle $C_v$ all vertices of which have degree 3. Let $H'$ denote this modification of $H$.
It is now straightforward to check that $\Gamma$ acts freely on $H$. (If some involution reverses an arc $R_s$, then we treat its midpoint as an intersection vertex of $H$ and apply to it the aforementioned blow-up operation.)

After this modification, we can indeed apply \Lr{babai} to $H'$: we choose a connected subgraph $D$ meeting each orbit at exactly one vertex, and contract each translate $gD, g\in \Gamma$ of $D$ into a vertex to obtain a \Cg\ $G$ of $\Gamma$.

Next, we remark that $G$ embeds in $X$. For this, pick a spanning tree $T$ of $D$, and a neighbourhood $U$ of $T$ in $X$ homeomorphic to $\R^2$ meeting no other translate of $D$. Note that $o\in U$ by the definitions. For each edge $xy$ incident with (the contraction of) $D$ in $G$, with $x\in V(D)$, say, we pick an \arc{o}{x} $R_x$ in $U$, making sure that $R_x$ is disjoint from $R_z$ for $z\neq x$. We can now think of $o$ rather than $D$ as a vertex of $G$. We repeat this at every other translate $g D$ of $D$ by mapping $U$ to $gU$. The union of the translates of the arcs $R_x$ with the appropriate subarcs of the original arcs $R_s$ defining $H$  defines an embedding of $G$ in $X$. 

If $x\in X$ were an accumulation point of $\sigma(V(G))$, then any neighbourhood $U_x\ni x$ would contain infinitely many elements of the orbit of $o$, violating our definition \ref{pd iii} of a \pd\ action. Thus  $\sigma(V(G))$ has no accumulation point in $X$. 

Easily, the canonical action $\Gamma \act E(G)$ extends to the original action $\Gamma \act H$ obtained by restricting $\Gamma \act X$. Thus $\Gamma \act G$ extends to $\Gamma \act X$.

\medskip
Moreover, we claim that if $X \subseteq \BS^2$ then our embedding $\sigma$, which we can then think of as an embedding into $\BS^2$, is \cova. For this, recall that each $g\in \Gamma$ is a homeomorphism of $X$, which by \Lr{Freu X} extends into a homeomorphism of $\BS^2$. Since 
$\Gamma \act H$ is defined by restricting $\Gamma \act X$, it follows that $g$ maps every facial walk of $H$ into a facial walk. 

As $G$ was defined by contracting $\Gamma$-translates of a subgraph of $H$, it is now easy to see that $g$ maps every facial walk of $G$ with respect to $\sigma$ into a facial walk, and so $\sigma$ is \cova.

Similarly, if $\Gamma \act X$ was orientation-preserving, then by considering the orbit of the edges of $o$ in \g under $\Gamma \act X$ it is easy to see that $\sigma$ is orientation-preserving too.
\end{proof}

{\it Remark:} \Lr{lem fin fac} below is proved using similar ideas, and the reader is advised to read that lemma right after reading the proof of \Lr{lem vapf}.

\section{Finiteness of face-boundaries} \label{sec fin fac}

In this section we prove the implication \ref{T ii} $\to$ \ref{T iv} of \Tr{tfae}:

\begin{lemma}\label{lem fin fac}
If a group $\Gamma$ admits a finitely generated, \covaly\ planar \Cg, then $\Gamma$ admits a finitely generated, \covaly\ planar Cayley multi-graph with no infinite face-boundaries.
\end{lemma} 

I do not know how to prove this without using \Lr{pd ext}, which is the essence of the implication \ref{T ii} $\to$ \ref{T iii}. \smallskip

We will need the following terminology.
If $\{V_1,V_2\}$ is a bipartition of the vertex set $V(G)$ of a graph \G, then the set of edges with exactly one endvertex in each $V_i$ is called a  \defi{cut} of \G.
The \defi{cut space} $\cb(G)$  is the vector space over the 2-element field generated by the finite cuts of \G; see \cite[\S 1.9]{diestelBook05} for a precise definition.

\begin{lemma}\label{cutspace}
Let $G= Cay(\Gamma,S)$ be a finitely generated planar \Cg. Then there is a finite set $D$ of cuts of \g \st\ $\bigcup \Gamma D$ generates the cut space of \G.
\end{lemma}
\begin{proof}
By a theorem of Droms \cite[Theorem 5.1]{DroInf}, every group $\Gamma$ that admits a finitely generated planar \Cg\ is finitely presented. By a theorem of Dunwoody \cite{DunAcc}, every finitely presented $\Gamma$ is accessible. The fact that every accessible group satisfies the conclusion of our statement is proved in \cite[Corollary~IV.7.6]{dicks_dunw}.
\end{proof}

\begin{proof}[Proof of \Lr{lem fin fac}]

Let \g be a finitely generated, \covaly\ planar \Cg\ of  $\Gamma$.
To simplify our exposition, let us first assume that \g is 3-connected; we will later employ \Lr{ext 3cd} to treat the general case.

By \Lr{FG emb k1} we obtain an embedding $\sigma: |G|\to \BS^2$, and by \Lr{pd ext} the canonical action $\Gamma\act G$ extends into a \pd\ action\\ $\Gamma\act  \BS^2 \sm \sigma(\OO(G))$.

Our strategy for modifying \g into a planar \Cg\ of $\Gamma$ with no infinite face-boundaries can be sketched as follows. In each  face $F$ of \g with infinite boundary we embed some arcs in a $\Gamma$-invariant way in order to split $F$ into faces with finite boundaries. We choose those arcs so that treating their intersection points as vertices defines a plane supergraph $G^+$ of $G$ using the ideas of \Lr{lem vapf}. Then we apply Babai's contraction \Lr{babai} to the action of $\Gamma$ on $G^+$, to contract the latter onto a \Cg\ $H$ of $\Gamma$ inheriting the property that all faces have finite boundaries.
\medskip

To make this sketch precise, let $D$ be a finite set of cuts of \g\ \st\ $\bigcup \Gamma D$ generates $\cb(G)$, provided by \Lr{cutspace}.
If $\sigma$ has no infinite face-boundary then there is nothing to prove. Otherwise, for every $B\in D$, and any two edges $e,f\in B$ lying in a common face-boundary $\partial F$ of a face $F$ of $\sigma$, choose an arc $A_{e,f,F}$ in $\cls{F}$ joining an endvertex of $e$ to an endvertex of $f$, which arc exists because $\partial F$ is a simple closed curve by \Lr{scc} and so $\cls{F}$ is homeomorphic to a closed disc by the Jordan-Sch\"onflies theorem. (If $e,f$ lie in more than one such $\partial F$,  choose such an arc in each $F$.) Let $\ca$ denote the union of the $\Gamma$-orbits of all such arcs $A_{e,f,F}$ with $e,f\in B\in D$.
If these arcs are chosen appropriately, then as in the proof of \Lr{lem vapf}, the union of $\ca$ with \g is a plane graph $G^+$, which is a supergraph of \G, embedded in $\BS^2$ via an extension $\rho$ of $\sigma(G)$. Here, we used again the fact that any $A_{e,f,F}\in \ca$  intersects  only finitely many elements of $\ca$ since $\Gamma \act \BS^2 \sm \sigma(\OO(G))$ is \pd. This fact easily implies that the identity map defines a homeomorphism $h: \OO(G) \to \OO(G^+)$, and furthermore the extension $\rho': |G^+| \to \BS^2$ of $\rho$ obtained by mapping each $\oo\in \OO(G^+)$ to $\sigma(h^{-1}(\oo))$ is an embedding.

 By construction, the action $\Gamma \act \BS^2 \sm \sigma(\OO(G))$ defines an action $\Gamma \act G^+$.
As in  \Lr{lem vapf}, by blowing up each vertex of $G^+ \sm G$ into a circle if needed, we may assume that the latter action is free. By Babai's contraction \Lr{babai}, there is a connected subgraph $\Delta$ of $G^+$ \st\ $H:= G^+ / \Delta$ is a \Cg\ of $\Gamma$. Here, we keep any parallel edges and loops of $H$ resulting from these contractions, that is, $H$ is a multi-graph, because they could be needed to retain the finiteness of face boundaries.

It is easy to see that $\Delta$ is finite,  because as noted above every arc $A_{e,f,F}$ as in the definition of $G^+$ is intersected by only finitely many translates of such arcs.

As in  \Lr{lem vapf}, we can modify $\sigma$ into an embedding $\tau: H \to \BS^2$ so that every face boundary of $H$ can be obtained from one of  $G^+$ by contracting each maximal subpath contained in a translate of $\Delta$ into a vertex. We claim that all face-boundaries of $H$ in $\tau$ are finite, which proves our statement. By the previous remark, this claim will follow if we can show that $G^+$ has no infinite face-boundary, or equivalently, that $G^+$ has no facial double-ray.

Suppose, to the contrary, that $R$ is a facial double-ray of $G^+$, and let $R'$ be the corresponding facial double-ray of $H$, obtained by contracting each maximal subpath contained in a translate of $\Delta$. We distinguish the following two cases.\\ 
{\em Case 1:} If $R$ comprises two disjoint sub-rays belonging to the same end of $G^+$, then as $\Delta$ is finite, this situation passes on to $H$: the two tails of $R'$ belong to the same end of $H$. This however is ruled out by a result of Kr\"on \cite{KroInf}, stating that in any almost transitive, plane graph with finite vertex degrees, disjoint tails of a facial double-ray lie in distinct ends.\\
{\em Case 2:} If $R$ has two disjoint sub-rays belonging to distinct ends $\oo,\chi$ of $G^+$, then let $F$ denote the face of $|G|$ \wrt\ $\sigma$ containing $R$, which face exists since $G^+$ is a supergraph of \G, and $\rho(R)$ avoids $\sigma(\OO(G))$. 
Then $\oo,\chi \in \partial F$ because $\rho'$ is an embedding of $|G^+|$ into $\BS^2$ as we saw earlier.
We claim that 
\labtequ{A seps}{\ti\ an arc $A\in \ca$  \st\ $\oo,\chi $ lie in distinct components of $\cls{F} \sm A$.}

Before proving this claim, let us see why it leads to a contradiction. The closure $\cls{D}$ of $D$ in $|G^+|$ is an \arc{\oo}{\chi}  by the definitions, and so $\rho'(\cls{D})$ is an \arc{\oo}{\chi} in $\cls{F}$. Thus claim \eqref{A seps} means that $D$ crosses $A$, and so it cannot be facial. This contradiction proves that Case~2 cannot occur either.

It remains to prove \eqref{A seps}. For this, note that the cutspace $\cb(G)$ contains a cut $B$ separating $\oo$ from $\chi$. Thus its generating set $\bigcup \Gamma D$ must contain such a cut $g C$ with $C\in D$. We claim that $g C$ contains edges $e_1,e_2$ in $\partial F$ \st\ $\oo,\chi $ lie in distinct components of  $\partial F \sm \{e,f\}$. This is true because $\partial F$ is a simple closed curve contained in $\sigma(|G|)$ by \Lr{scc}, and so its preimage $Z:= \sigma^{-1}(\partial F)$ is a homeomorph  of $\BS^1$ in $|G|$. In particular, $Z$ contains two internally disjoint ${\oo}$--${\chi}$~arcs $Z_1,Z_2$. But then the cut $g C$ must meet both those arcs in order to separate $\oo$ from $\chi$, because \fe\ \arc{\oo}{\chi} $Y$ and every cut $J$ separating $\oo$ from $\chi$ we have $Y \cap J \neq \emptyset$ by \Lr{cut arc}. Thus we can let $e_i$ be any edge in $g C \cap \sigma(Z_i)$ for $i=1,2$.
Finally, by the definition of $G^+$, the latter contains an arc $A$ in $\cls{F}$ from an endvertex of $e_1$ to an endvertex of $e_2$. This arc disconnects $\oo,\chi $ in $\cls{F}$ because $e_1,e_2$ disconnect $\oo,\chi $ in $\partial F$.
\medskip

This completes our proof in the case where \g is 3-connected. If it is not, then we apply \Lr{ext 3cd} to embed \g into a planar 3-connected plane supergraph $G'$ to which the action of $\Gamma$ extends \covaly. We then repeat the above construction verbatim with $G$ replaced by $G'$. The only point that requires some care is to extend the conclusion of \Lr{cutspace} to $G'$. This is straightforward: we apply \Lr{cutspace} to \g to obtain a  set of cuts $D$. For each $B\in D$, we let $B'$ be the cut of $G'$ obtained by adding, \fe\ $e\in B$, the up to two `parallel' copies of $e$ from any ladders we attached to face-boundaries containing $e$ in the proof of \Lr{ext 3cd}. Moreover, for each of the finitely many $\Gamma$-orbits $O$ of vertices of $G'$, we pick a representative $v_O\in O$, and let $B_O$ be the cut comprising the edges incident with $v_O$. Then the collection $D':= \{B' \mid B\in D\} \cup \bigcup_{O\in V(G')/\Gamma} B_O$ of all these cuts of $G'$ is finite, and their translates  generate $\cb(G')$ as the reader can easily check.
\end{proof}

I do not know if \Lr{lem fin fac} remains true if one forbids loops and/or parallel edges:
\begin{problem}
Suppose $\Gamma$ admits a finitely generated, \covaly\ planar \Cg. Must $\Gamma$ admit a finitely generated, \covaly\ planar Cayley graph with no infinite face-boundaries?
\end{problem}

\section{Proof of \Tr{tfae} and \Cr{Kleinian}} \label{proof}

We now put the above results together to prove the main results of the paper:
\begin{proof}[Proof of \Tr{tfae}]
For  the implication\ref{T ii} $\to$ \ref{T iii}, let $G$ be a \covaly\ planar, finitely generated, Cayley graph of a group $\Gamma$. By \Lr{FG emb k1} there is a \cova\ embedding $\sigma: |G| \to \BS^2$. Applying \Lr{pd ext} we obtain a \pd\ action $\Gamma \act  (\BS^2 \sm \sigma(\OO(G))$ extending $\Gamma \act G$. Since $\Gamma \act G$ is faithful, so is $\Gamma \act  (\BS^2 \sm \sigma(\OO(G))$. We claim that $\BS^2 \sm \sigma(\OO(G))$ is one of the four 2-manifolds  in the statement of \ref{T iii}. This is indeed the case, because of the well-known fact, first observed by Hopf \cite{Hopf}, that $\OO(G)$ is either homeomorphic to the Cantor set or it contains at most 2 points. 


\medskip
For the implication \ref{T i} $\to$ \ref{T ii}, let $\Gamma \act X$ be a \pd\ faithful action where $X\subseteq \BS^2$ is 
 a 2-manifold. Then $X$ is uniquely planar by \Lr{Freu X}, and so \Lr{lem vapf}  yields the desired \Cg\ $G$. 

\medskip
The implication \ref{T ii} $\to$ \ref{T iv} is \Lr{lem fin fac}.

\medskip 
As \ref{T iii} $\to$ \ref{T i} and \ref{T iv} $\to$ \ref{T ii} are trivial, we have proved the equivalence of all four conditions.
\end{proof}

\begin{theorem}[{\cite[p.~299]{Maskit}}] \label{thm Maskit}
Let $p:\tilde{S} \to S$ be a regular covering of the (topologicaly) finite (Riemann) surface $S$, where $\tilde{S} $ is planar. Let $F$ be the group of deck transformations on $\tilde{S}$. There is a Koebe group $\Gamma$, with invariant component $D$, and there is a homeomorphism $h:\tilde{S} \to D$, so that $h_*: F \to \Gamma$ is an isomorphism.
\end{theorem}

Here, a surface $S$ is \defi{topologicaly finite}, if it is homeomorphic to the interior of a compact 2-manifold with or without boundary.
I will not repeat all the details (which can be found e.g.\ in \cite{Maskit}) of the definition of \defi{Koebe group}, as they are not needed in this paper. What matters for us is that every Koebe group is a function group by definition.

\begin{proof}[Proof of \Cr{Kleinian}]
The backward direction is obvious: if $\Gamma$ is a function group, then its action on an invariant component of its domain of discontinuity is as desired.

For the forward direction, let  $\Gamma \act X$ be a faithful, \pd\  action by orientation-preserving homeomorphisms on the surface $X\subseteq \BS^2$. We would like to apply \Tr{thm Maskit} with $\tilde{S} =X$ and $S= X/\Gamma$, but $\Gamma \act X$ need not be free, and $X/\Gamma$ need not be topologically finite. Therefore, we will use the above results to find a better action of $\Gamma$.

Indeed, the implication \ref{T i} $\to$ \ref{T iv} of \Tr{tfae} yields a planar \Cg\ \g of $\Gamma$, with a $\Gamma$-\cova\ embedding $\sigma: |G| \to \BS^2$ every face of which is bounded by a cycle. Moreover, when applying  \Lr{lem vapf} to prove the implication \ref{T i} $\to$ \ref{T ii}, we use the last statement of \Lr{lem vapf} to ensure that $\sigma$ is orientation-preserving. 
We then apply the last statement of \Lr{pd ext} to obtain a co-compact, \pd\ action $\Gamma \act Y$ with $Y:= \BS^2 \sm (\sigma(\OO(G)))$, where the set $Z$ of points of $Y$ with non-trivial stabiliser contains at most one point from the interior of each face of $\sigma$. Restricting $\Gamma \act Y$ to $X' := Y \sm Z$ thus yields a free action, in other words, $X'$ regularly covers $S:= (Y \sm Z)/\Gamma$. Since $\Gamma \act Y$ is co-compact, it is easy to see that $S$ is topologically finite. Thus we can apply  \Tr{thm Maskit} to deduce that $\Gamma$ is a Koebe group, and in particular a function group.
\end{proof}

As a corollary, we deduce that the finiteness condition in \Tr{thm Maskit} can be dropped if one only wants an algebraic rather than a `geometric' isomorphism between $F$ and some Koebe group (the finiteness condition cannot be dropped in general without this relaxation):

\begin{corollary} \label{cor Maskit}
Let $p:\tilde{S} \to S$ be a regular covering of a surface $S$, where $\tilde{S} $ is planar. Let $F$ be the group of deck transformations on $\tilde{S}$. Then $F$ is isomorphic to a Koebe group (equivalently, to a function group).
\end{corollary}
\begin{proof}
By the definition of deck transformations, $F$ acts freely by homeomorphisms on $\tilde{S}$, and by the definition of a covering this action is \pd. Thus we can apply \Cr{Kleinian} to deduce that $F$ is isomorphic to a Koebe group.
\end{proof}

\subsection{Geometrizing the action} \label{sec geom}
As mentioned in the introduction, any action $\Gamma \act X$ as in \ref{T iii} can be `geometrised' to obtain an 
 action by isometries on a smooth manifold with the same properties. This can be proved as follows.
Recall that every 2-manifold is homeomorphic to a smooth manifold \cite{MoiseGT}. Thus it just remains to endow $X$ with a $\Gamma$-invariant metric $d$. A standard way to do this (which I learnt from a mathoveflow post\footnote{\url{https://mathoverflow.net/questions/251627/proper-discontinuity-and-existence-of-a-fundamental-domain}} by Misha Kapovich.) is by endowing $X$ with a $\Gamma$-invariant Riemannian metric $g$, and letting $d$ be the corresponding distance function.  This $g$ can be constructed by first constructing an arbitrary Riemannian metric on the quotient orbifold $X/\Gamma$, e.g.\ using a partition of unity, and then lifting $g$ to $X$. Here, we used the well-known observation of Thurston that for every \pd\ action on a manifold, the quotient space is an orbifold, see e.g.\ \cite[Proposition~20]{BorRie}.

\section{Determining $X$} \label{X}

The main aim of this section is to prove \Tr{ext pd i}.\smallskip

We say that an end $\oo\in \OO(X)$ is an \defi{accumulation end} \wrt\ an action $\Gamma \act X$, if \oo\ is an accumulation point of some orbit $\Gamma x, x\in X$ for the extension $\Gamma \act |X|$ of the action to the \FC\ of $X$. 

Given topological spaces $Y\subset X$, we write $\lim Y$ for the set of accumulation points of $Y$ in $X$.

\begin{proposition}\label{def acc end}
Let $X$ be an arc-connected locally compact space, $\Gamma \act X$ a \pd\ action, and $\oo\in \OO(X)$. Then the following are equivalent:
\begin{enumerate}
\item \label{ai} \ti\ $x\in X$ \st\ $\oo \in \lim \Gamma x$;
\item \label{aiii} \ti\ a compact $K\subset X$ \st\ $\oo \in \lim \Gamma K$;
\item \label{aiv} \fe\ compact $L\subset X$, we have $\oo \in \lim \Gamma L$; and 
\item \label{aii} \fe\ $x\in X$, we have $\oo \in \lim \Gamma x$.
\end{enumerate}
\end{proposition} 
\begin{proof}
The implications \ref{ai} $\to$ \ref{aiii} and  \ref{aiv} $\to$ \ref{aii} are trivial since points are compact. So is \ref{aii} $\to$ \ref{ai}. To show \ref{aiii} $\to$ \ref{aiv}, let $L$ be any compact subset of $X$, let $A$ be an \arc{K}{L}, and let $K':= K\cup A \cup L$. Note that $K'$ is compact. Let $U$ be a basic open neighbourhood of $\oo$ in $|X|$, and recall that $U\cap X$ is a component of $X \sm S$ for some compact $S\subset X$. Since $\oo \in \lim \Gamma K$, there are infinitely many elements of $\Gamma K$, and hence of $\Gamma K'$ meeting $U$. By definition \ref{pd ii} of a \pd\ action, at most finitely many of these elements meet the compact set $S$, and therefore infinitely many of them are contained in $U$. In particular, $U$ meets $\Gamma L$. As this holds for any  basic open neighbourhood $U$ of $\oo$, we deduce that $\oo \in \lim \Gamma L$ as claimed.
\end{proof}

Let $\OO^\Lambda(X)\subseteq \OO(X)$ denote the set of accumulation ends of $X$ \wrt\ $\Gamma \act X$. In other words, $\OO^\Lambda(X)$ is the limit set of the extension of the action $\Gamma \act X$ to $|X|$.

The proof of the following statement is standard; the main idea goes back to Hopf \cite{Hopf}.

\begin{theorem}\label{acc ends}
Let $\Gamma \act X$ be a \pd\ action on an arc-connected metrizable space $X$. 
If the space $\OO^\Lambda(X)\subseteq \OO(X)$ of  accumulation ends contains more than 2 points, then it is homeomorphic to the Cantor set.
\end{theorem} 
\begin{proof} 	\mymargin{DOESN'T WORK: Alternative proof strategy: prove that the extension of $\Gamma \act X$ to $X'$ has a compact quotient (for otherwise any end would be a non-accumulation end of $X'$, which don't exist) and is therefore co-compact NO: SEE EXAMPLE IN REMARK BELOW. Then apply Hopf's main theorem, using \Tr{ext pd}.}
Suppose $\OO^\Lambda(X)$ contains three distinct accumulation ends $\oo,\chi,\psi$. We will show that none of them is an isolated point of $\OO^\Lambda(X)$.
Let $K\subset X$ be a compact set separating $X$ into subspaces $U_\oo \ni \oo, U_\chi \ni \chi, U_\psi \ni \psi$. We may assume that $K$ is connected, for otherwise we can enlarge it by a set of arcs joining its finitely many components.


Suppose, for a contradiction, that $\oo$ has a neighbourhood $U$ containing no other end in $\OO^\Lambda(X)$. By choosing a larger $K$ if needed, we may assume that $U= U_\oo$. 
Pick some $x\in K$ and $g\in \Gamma$ \st\ $gx \in U$ and $gK \cap K= \emptyset$, which exists since $U$ contains an accumulation end and $\Gamma \act X$ is \pd. Then $gK\subset U$ because $gK$ is connected and it meets $U$. Therefore, $gK$ disconnects no two ends of $X$ living outside $U_\oo$, because any two such ends live in the connected space $X \sm U$. But $gK$ disconnects $g\oo,g\chi,g\psi$ from each other, and so at least two of them live in $U$, contrary to our assumption.

This proves that if $\OO^\Lambda(X)$ contains more than two points, then it contains no isolated point. It is easy to check that $\OO^\Lambda(X)$ is closed (therefore compact since $\OO^\Lambda(X)\subset |X|$). As it is totally disconnected, it is homeomorphic to the Cantor set by \Prr{Cantor}.
\end{proof}


\begin{corollary}\label{cor four}
Let $\Gamma \act X$ be a faithful, \pd\ action on a 2-manifold $X\subseteq \BS^2$ \st\ all ends of $X$ are accumulation ends. Then $X$ is homeomorphic to the sphere, the plane, the open annulus, or the Cantor sphere.
\end{corollary} 
\begin{proof} 	
By \Lr{acc ends}, either $\OO(X)$ contains at most 2 ends, or it is homeomorphic to the Cantor set. By \Tr{Richards}, $X$ is homeomorphic to the sphere, the plane, the open annulus, or the Cantor sphere, if it has 0, 1, 2, or a Cantor space of ends, respectively.
\end{proof}

{\bf Remark}: The assumption that all ends of $X$ are accumulation ends does not imply that the action is co-compact, or that $|\OO(X)|=|\OO(\Gamma)|$; consider for example the action of $\Z$ on $\R^2$ by addition in one coordinate. 

This example is also relevant in \Tr{ext pd i} from the introduction, which we restate here for convenience:

\begin{theorem}\label{ext pd} 
Let $\Gamma \act X$ be a \pd\ action on a metrizable, arc-connected, locally compact space $X$. Then the canonical extension $\Gamma \act (X \cup \OO^V(X))$ of the action to the non-accumulation ends of $X$ is \pd. 
\end{theorem} 

For the proof of this we need to introduce the following notions and a lemma.

We say that a subspace $K\subset X$ is a \defi{separator} of a topological space $X$, if $X\sm K$ can be written as the union $U\cup U^c$ of two disjoint, non-empty, open subspaces. Note that $U,U^c$ are also closed in $X\sm K$ in this case. In this case, we say that $U,U^c$ are \defi{sides} of $K$. We remark that $K$ does not uniquely determine its sides if  $X\sm K$ has more than two components. 
The following lemma is a variant of \cite[Theorem 1, Chapter V,\S 46, VIII]{Kuratowski}\footnote{I would like to thank Max Pitz for this reference.}. We provide a proof for convenience.

\begin{lemma}\label{nested}
Let $X$ be an arc-connected, metrizable topological space, and\\ $K_Y,K_Z \subset X$ two disjoint closed and connected separators of $X$ with sides $Y,Y^c$ and $Z,Z^c$, respectively. Then at least one of the sets $Y\cap Z, Y \cap Z^c, Y^c\cap Z, Y^c \cap Z^c$ is empty.
\end{lemma} 
\begin{proof} 	
Since $K_Y,K_Z$ are connected, and  disjoint, each of them must be contained in a side of the other. We may assume \obda\ that $K_Z \subseteq Y$ and $K_Y \subseteq Z$ as the two sides of each separator are interchangeable. 

Suppose, for a contradiction, that all our sets $Y\cap Z, Y \cap Z^c, Y^c\cap Z, Y^c \cap Z^c$ are non-empty, and pick two points $p\in Y^c \cap Z^c, q\in Y\cap Z$. Let $A$ be a \arc{p}{q}\ in $X$, and let $x$ be the first point of $A$ in $K_Y \cup K_Z$, which exists since $K_Y \cup K_Z$ is closed. If $x\in K_Y$, then the sub-arc of $A$ from $p$ to $x$ is a \arc{Z^c}{Z} in $X \sm K_Z$. Otherwise, we have $x\in K_Z$, and the sub-arc of $A$ from $p$ to $x$ is a \arc{Y^c}{Y} in $X \sm K_Y$. In both cases we obtain a contradiction as we have split an arc as a union of two disjoint open sets.
\end{proof}

\begin{proof}[Proof of \Tr{ext pd}]
Consider the subspace $X':= X \cup \OO^V(X)$ of $|X|$. To show that 
  the canonical extension $\Gamma \act X'$ of $\Gamma \act X$ is \pd, we will show that \fe\ $x,y\in X'$ there are open neighbourhoods $O_x\ni x , O_y\ni y$ \st\ $\{ g\in \Gamma \mid gU_y \cap U_x \neq \emptyset\}$ is finite and apply definition \ref{pd iii} of a \pd\ action. 
  
We claim that there is a sequence \seq{K^x} of compact, connected separators of $X'$ with corresponding sides $(U_n, U^c_n)$ \st\ \seq{U}\ is a base of open neighbourhoods of $x$ (which we will later plug into \Lr{nested}). Indeed, if $x\in \OO^V(X)$ we can let \seq{U} consist of basic open neighbourhoods of $x$, and let $K^x_n$ their associated compact sets. We can assume each $K^x_n$ is connected because otherwise we can join its finitely many components with arcs of $X$. If $x\in X$, then \ti\ a local base \seq{C}\ of compact neighbourhoods of $x$ since $X$ is locally compact. Let $\{\partial C_n\}$ be their frontiers, let  $K^x_n$ be a compact connected set containing $\partial C_n$, and let $U_n:= C_n \sm K^x_n$. Then $U_n$ is open in $X' \sm K^x_n$ as desired, and so is $U^c_n:= X' \sm (K^x_n \cup U_n)$. Replacing $x$ by $y$ throughout, we analogously define $\seq{K^y}$ and the sides $\{V_n,V^c_n\}$.
 
 As $\Gamma \act X$ is \pd, the set $\{ g\in \Gamma \mid g K^x_i \cap K^y_j \neq \emptyset\}$ is finite \fe\ $i,j\in \N$. Thus we can apply \Lr{nested} to the separators  $g K^x_i , K^y_j$ of $X'$ \fe\ $i,j\in \N$ and all but finitely many $g\in \Gamma$. If the transporter set $(U_i, V_j)$ is finite for some $i,j$, then we are done by \ref{pd iii}. If not, then either $|(U_i \mid K^y)|=\infty$ for every $i$, or $|(V_i \mid K^x)|=\infty$ for every $i$. If the former is the case, then $x\in \lim K^y$, and if the latter is the case, then $y\in \lim K^x$. This leads to a contradiction as $X'$ contains no accumulation end, and no point of $X$ is an accumulation point of an orbit under $\Gamma \act X$ (here we tacitly used  \Prr{def acc end}).

\end{proof}

\section{Relationship to planar  groups and planar discontinuous groups} \label{sec rels}


Droms et.\ al.\ \cite{DrSeSeCon} provided an example of a \Cg\ $G$ 
which is planar but admits no embedding in $\BS^2$ for which the natural action on \g by its group is realized by homeomorphisms of $\BS^2$. Here, we provide such an example with the stronger property that its group does not admit any faithful action on $\BS^2$ (or any 2-manifold $X\subseteq \BS^2$) by homeomorphisms.
\medskip

For this, we will find finite groups $R,P$ of homeomorphism of $\BS^2$, containing involutions $b_r,b_p$ respectively, \st\ in any faithful action $R \act \BS^2$ the homeomorphism $b_r$ reverses the orientation of $\BS^2$, and in any faithful action $P \act \BS^2$ the homeomorphism $b_p$ preserves the orientation. In addition, we will display planar \Cg s $G_R,G_P$ of the two groups in which $b_r$ and $b_p$ are contained in the generating set, and thus appear as edges. By amalgamating any two such groups $R,P$ along the 2-element subgroups spanned by $b_r,b_p$ respectively, we obtain a group with a planar \Cg\ $G$, which can be obtained by embedding infinitely many copies of $G_R$ and $G_P$ glued along their amalgamated edges corresponding to $b_r,b_p$.
\medskip

We can choose $R$ to be the alternating group $A_4$. This group has presentations $\left< r,g \mid r^3,g^3,(rg)^2 \right>$ and 
$\left< k,r | k^2, r^3, (kr)^3\right>$, where $r=(1,2,3), g=(2,3,4), k=(1,3)(2,4)$ in permutation cycle notation. The first presentation shows that $R$ is generated by elements of odd order. Easily, every element of odd order preserves the orientation in any action $R \act \BS^2$. These two observations combined yield that every element of $R$  preserves the orientation in any action $R \act \BS^2$. 

The second presentation gives rise to a planar \Cg\ of $R$, namely the truncated tetrahedron\footnote{See {\url{http://weddslist.com/groups/cayley-plat/index.html}} for a figure.}, and its generator $k$ is an involution. We can thus let $G_R$ be this \Cg\ and $b_r=k$.

\epsfxsize=0.4\hsize
\showFig{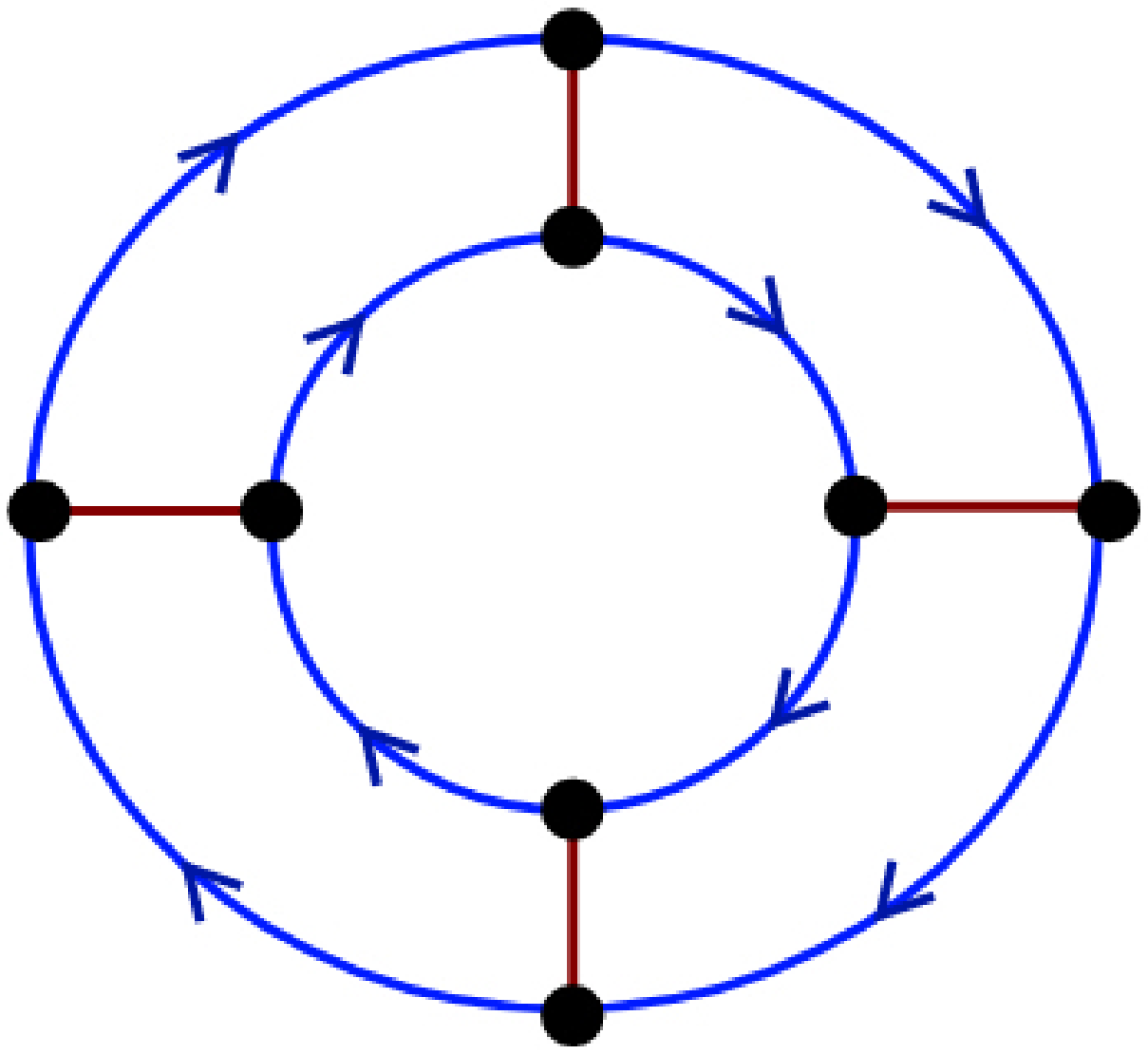}{The standard \Cg\ of $\Z_4 \times \Z_2$.}

We can choose $P$ to be the group $\Z_4 \times \Z_2$. We claim that every planar  \Cg\ $G= Cay(P,S)$ of this group has a subgraph $G' = Cay(P,S')$ with $S'\subseteq S$ isomorphic to the \defi{prism} depicted in \fig{prism}. To see this, note that $S$ must contain the element $(1,0)$ or its inverse $(3,0)$ because these elements are not contained in the span of the remaining elements of $P$ by a parity argument. If $S$ also contains any of the elements $(1,1)$ or $(3,1)$, then $G$ contains the graph of \fig{K44} as a subgraph. But that graph is not planar, because it is isomorphic to the complete bipartite graph $K_{4,4}$, which contains the Kuratowski graph $K_{3,3}$ as a subgraph. Thus $S$ also contains one the elements $(0,1)$ or $(2,1)$.
It is easy to check that any  choice $S'$ of two generators as above yields a \Cg\  $G' \subseteq G$ isomorphic to the graph of \fig{prism}. (In fact, it is not hard to show that $G = G'$, but we will not need this.) 

\epsfxsize=0.4\hsize
\showFig{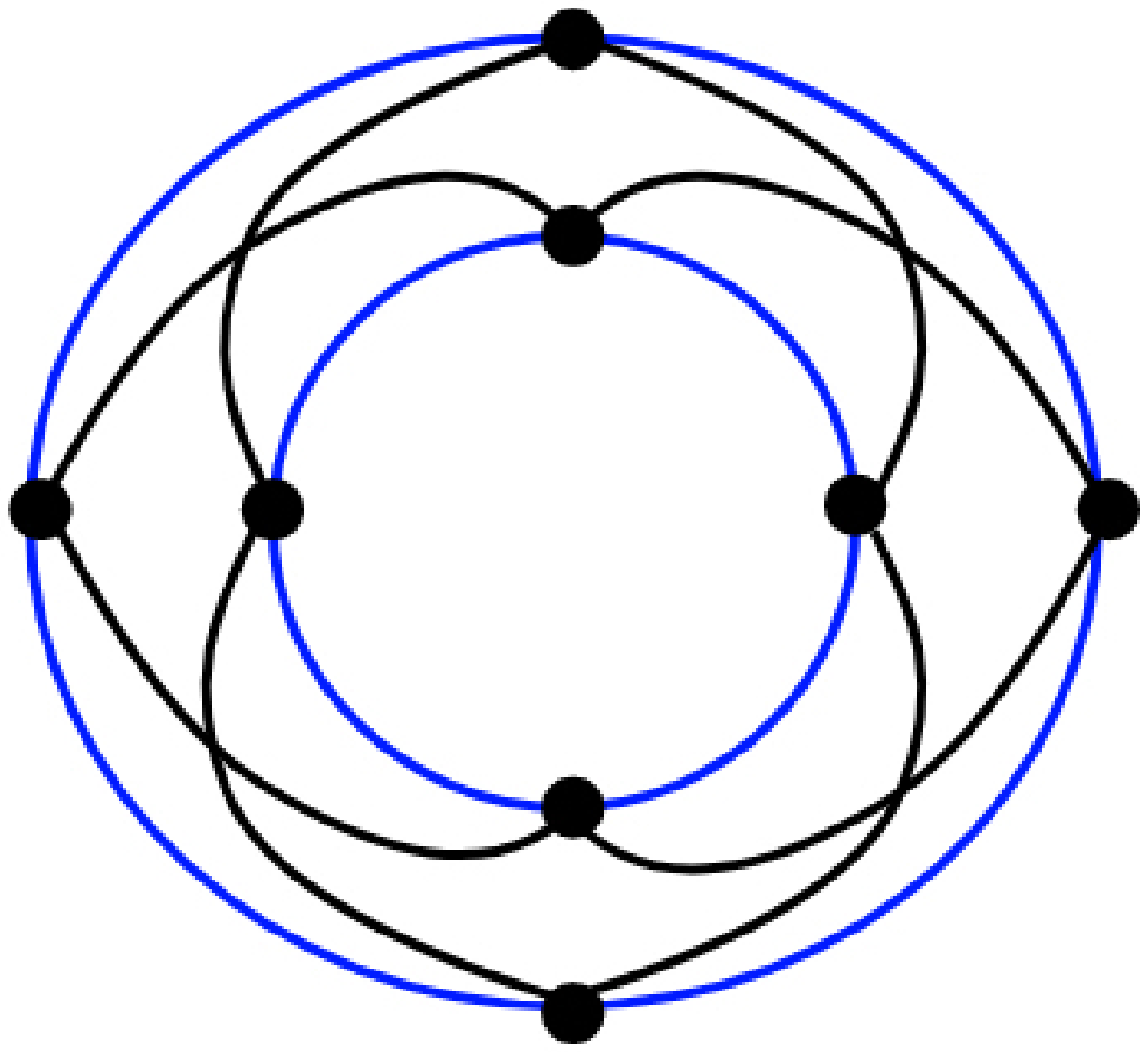}{A non-planar \Cg\ of $\Z_4 \times \Z_2$. This graph is the complete bipartite graph $K_{4,4}$: the vertices at the same horizontal/vertical level form the two partition classes.}

Applying \Lr{lem vapf} to any action $P\act \BS^2$, we obtain a planar  \Cg\ $H$ of $P$ \st\ the canonical action $P \act H$ extends to $P\act \BS^2$. By the above claim, $H$ contains the \Cg\ $G'$ of \fig{prism} as a subgraph. Note that $P \act G'$ extends to $P\act H$, and from there to $P \act \BS^2$. Recall that we would like to choose an involution $b_P\in P$ that reverses  orientation in any action $R \act \BS^2$. Note that $P$ contains exactly 3 involutions, namely $(2,0), (0,1), (2,1)$. Of these involutions, only the first one is contained in a subgroup of $P$ isomorphic to $\Z^4$ ($P$ has exactly two such subgroups). Let $b_P=(0,1)$. Then no matter which generating set of $P$ gives rise to the \Cg\ $G'$, our involution $b_p$ will correspond to a vertex of $G'$ that is not in the same monochromatic 4-cycle as the identity. But then the action of $b_p$ on $G'$ reverses the cyclic order of the edges incident with any vertex. Since $P \act G'$ extends to $P \act \BS^2$, this means that $b_p$ reverses the orientation of $\BS^2$ in $P \act \BS^2$, which was an arbitrary action of $P$ on $\BS^2$. Finaly we can just let $G_P:= G'$ be our choise of \Cg\ of $P$.
\medskip

Let now $\Gamma := R \star_{b_r=b_p} P$ be the amalgamation product of $R$ and $P$ over the subgroups spanned by $b_r, b_p$, respectively. Then $\Gamma$ cannot act faithfully on $\BS^2$, because by the above discussion its element $b_r=b_p$ would have to both preserve and reverse the orientation in any such action. Moreover, we can obtain a planar \Cg\ of $\Gamma$ by recursively glueing copies of $G_R$ and $G_P$ along their edges corresponding to $b_r,b_p$. This proves in particular

\begin{proposition}\label{subfam}
The groups satisfying the conditions of \Tr{tfae} form a proper subfamily of the groups admitting finitely generated \Cg s.
\end{proposition}

Next, we will show that the groups satisfying the conditions of \Tr{tfae} form a proper superfamily of the finitely generated planar discontinuous groups, i.e.\ the finitely generated groups acting faithfully and \pd ly on $\R^2$. For the expert reader this may be a straightforward consequence of \Cr{Kleinian}.

The example we will consider is $\Gamma:= \Z \times \Z_3$. The standard \Cg\ of $\Gamma$ is easily seen to be planar and 3-connected. Thus $\Gamma$ satisfies condition \ref{T ii} of \Tr{tfae}. It acts faithfully and \pd ly on the open annulus but not, as we will now show, on $\R^2$.

Suppose, for a contradiction, that $\Gamma\act \BS^2$ is such an action, and apply \Lr{lem vapf} to obtain a planar  \Cg\ $G$ of $\Gamma$ and an embedding $\sigma : G \to \R^2$ \st\ $\sigma(V(G))$ has no accumulation points in $\R^2$, and the canonical action $\Gamma \act G$ extends to $\Gamma \act \BS^2$.

Our first claim is that there is a cycle $C$  of $G$ fixed by the action of the subgroup $\Delta$ corresponding second factor $\Z_3$ in the definition of $\Gamma$. To see this, let $e,g,g^2$ denote the elements of $\Delta$, and let $A$ be an \pth{e}{g}\ in \G. Let $X:= A \cup gA \cup g^2 A$ be the subgraph of \g comprising its $\Delta$-translates. By Babai's \Lr{babai} applied to the (free) action of $\Delta$ on $X$, there is a contraction $X/F$ of $X$, with $F\subseteq X$ connected, with $X/F$ being a  \Cg\ of $\Delta$. The only  \Cg\ of the group $\Delta$ with 3 elements is the triangle. Let $B$ be a path in $F$ joining the endvertices of the two edges of $X/F$ incident with $F$. Then the 3 edges of $X/F$ combined with the 3 $\Delta$-translates of $B$ form the desired cycle $C$.

Next, we prove that there is a further cycle $D$  of $G$ fixed by the action of  $\Delta$ \st\ the closures of the insides of $\sigma(C)$ and $\sigma(D)$ are disjoint. For this, note that there are infinitely many translates $hC, h\in \Gamma$ of $C$, and only finitely many of them intersect $C$ by local finiteness. If $C, D=hC$ are disjoint, and any of them is embedded inside the other by $\sigma$, then by iterating $h$ or $h^{-1}$ we can obtain an infinite set of translates embedded inside $C$ because $\Gamma \act G$ extends to $\Gamma \act \BS^2$. But this would contradict the fact that $\sigma(V(G))$ has no accumulation points in $\R^2$, and so our claim is proved.

Let $P$ be a path from $C$ to $D$ in $G$ \st\ the interior of $P$ avoids $C \cup D$, which exists since $G$ is connected. If the three translates $P, gP, g^2P$ of $P$ via $\Delta$ are pairwise disjoint, then combined with $C$ and $D$ they decompose $\R^2$ into 6 domains (\fig{6regions}), exactly one of which is non-compact and delimited by two of $P, gP, g^2P$ and a subpath of each of $C,D$ (here, we are tacitly interpreting the aforementioned subgraphs of \g as subspaces of $\BS^2$ by using our embedding $\sigma$). But the three domains of this form are permuted by the action of $\Delta$, leading to a contradiction as no homeomorphism can map a compact set onto a non-compact set. 

\epsfxsize=0.6\hsize

\showFig{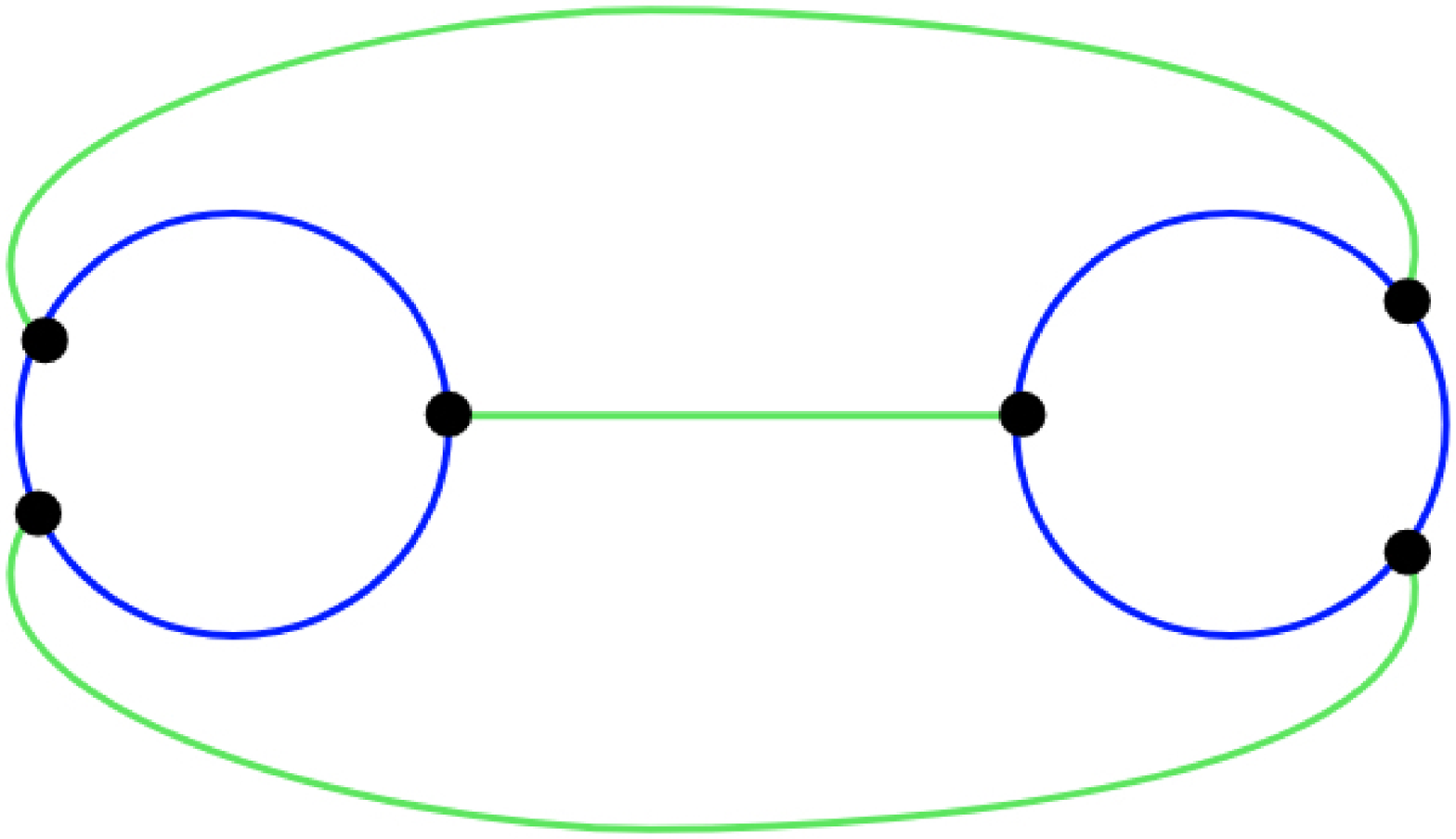}{The situation when $P, gP, g^2P$ (displayed in green, if colour is shown) are pairwise disjoint. We show that this situation can be achieved by choosing $P$ appropriately.}

Thus it remains to show that we can choose $P$ above so that its translates $P, gP, g^2P$ are pairwise disjoint. To see this, assume they are not, and notice that the subgraph $H:= P\cup gP \cup g^2 P$ of \g is connected in this case. Since $\Delta \subset \Gamma$ acts freely on \G, it does so on $H$. By Babai's \Lr{babai} again, there is a contraction $H/F$ of $F$, with $F\subseteq H$ connected, which is a \Cg\ of $\Delta$, and again $H/F$ can only be a triangle. Notice that $F$ must contain exactly one of the six endpoints of the three paths $P, gP, g^2P$ incident with $C$, and exactly one  incident with $D$ by the definitions. Let $x,y$ be those endpoints, and let $P'$ be a \pth{x}{y}\ contained in $F$. Then $P'$ is a path from $C$ to $D$ in $G$, and its translates $gP',g^2 P'$ are contained in the translates $gF,g^2F$ of $F$, and are hence disjoint to each other and to $P'$ by the choice of $F$. This completes our proof.


\comment{
\begin{proposition}\label{}
$D:= D_3 \star_{\Z_2} (\Z_3 \times \Z_2)$ does not admit any faithful action on $\BS^2$ by homeomorphisms.
\end{proposition} 
\begin{proof} 	
Suppose there is such an action $D \act \BS^2$. By restricting this action we obtain actions by its finite subgroups  $D_3 \act \BS^2$ and $\Z_3 \times \Z_2\act \BS^2$. Let $b$ be the non-trivial element of the subgroup $D_3  \cap (\Z_3 \times \Z_2)$ of $D$ (which is isomorphic to $\Z_2$). We claim that in any action $D_3 \act \BS^2$, the homeomorphism $b$ is orientation-preserving, while in any action $\Z_3 \times \Z_2 \act \BS^2$ it is orientation-reversing. This leads to a contradiction since $b$ can only do one of the two in $D \act \BS^2$.

Since all finite actions are \pd, we can apply  \Lr{lem vapf} to the two above actions $D_3 \act \BS^2$ and $\Z_3 \times \Z_2\act \BS^2$. This yields a planar \Cg\ $G$ of $D_3$ \st\ $D_3 \act G$ extends to $D_3 \act \BS^2$ and a planar \Cg\ $H$ of $\Z_3 \times \Z_2$ \st\ $\Z_3 \times \Z_2 \act H$ extends to $\Z_3 \times \Z_2 \act \BS^2$.

There are only 2 candidates for $G$ up to isomorphism: it is straightforward to check that $D_3$ has, up to inverses, exactly one non-involution, and such an element is needed in order to generate $G$. Moreover, if we admit all 3 involutions of $D_3$ as generators of $G$, then $G$ will contain the Kuratowski graph $K_{3,3}$, contradicting its planarity. Thus $G$ is generated by one non-involution and 1 or two involutions, and is hence isomorphic to the prism $P$ of order 6, or $P$ with three diagonals added. Both these graphs are 3-connected, and so the embedding of $G$ in $ \BS^2$ is essentially unique. Moreover, it is easy to check that each involution in $D_3$ preserves orientation when its action on $G$ is extended to $ \BS^2$. To summarise, we have proved that 

\end{proof}
}

\medskip
We finish this section with an example of a finitely generated Kleinian group that is not one of the groups of \Tr{tfae}; even more, it does not admit any planar \Cg. In fact most Kleinian groups have this property, but we present the following explicit example ---for which I thank B.~Bowditch (private communication)--- for the non-expert reader.

\begin{proposition}\label{ex klein}
There is a finitely generated Kleinian group that does not admit a planar \Cg.
\end{proposition} 
\begin{proof}[Proof (sketch)]
Let $S$ be a closed orientable surface of genus at least 1, and let $S'$ be obtained from $S$ by removing a topological disc bounded by a (contractible) simple closed curve $C$. Let $D$ be the space obtained by identifying four copies of $S'$ along $C$. It is shown in \cite[Proposition~5.1]{BowMes4dim} that $D$ admits a metric $d$ \st\ $(D,d)$ embeds (as a metric subspace) in a complete hyperbolic 3-manifold $W$ and $W$ retracts onto $D$. Thus $W$ and $D$ have the same fundamental group, and we let  $H:= \pi_1(D) \isom  \pi_1(W)$. It is well-known that the fundamental group of every  complete hyperbolic 3-manifold, in particular $H$, is Kleinian. 

By the Seifert--van~Kampen theorem, $H$ is the amalgamation product of four free groups of rank 2 (or two copies of the fundamental group of the double-torus) along an infinite cyclic subgroup. This remark allows us to visualise its canonical \Cg\ as a union of regular tilings of the plane, glued along certain common lines. It is not hard to deduce from this that $H$ is 1-ended, and torsion-free. 
Since 1-ended \Cg s are 3-connected \cite[Lemma~2.4]{baGro}, if $H$ admits a planar \Cg\ \G, then \g has an essentially unique embedding into $\BS^2$ by Whitney's \Tr{imrcb}. Then \ref{T iii} of \Tr{tfae} implies that $H$ is the fundamental group of a compact 2-orbifold $M$, and since $H$ is torsion-free $M$ is a manifold, i.e.\ a closed orientable surface. It is an exercise to show that $\pi_1(D)$ is not the fundamental group of a closed orientable surface, for example by noticing that the aforementioned \Cg\ is not quasi-isometric to any regular tiling of the plane.
\end{proof}

\section{Relationship to orbifold fundamental groups} \label{sec orb}

Let $\Gamma \act X$ be an action as in \Tr{tfae}~\ref{T iii}. Then the quotient space $O:= X/\Gamma$ is a good compact 2-orbifold (see e.g.\ \cite[Proposition~20]{BorRie}). It follows from the standard theory of covering spaces (\cite[Proposition~1.40]{Hatcher}) that 
\labtequ{orb pi}{$\Gamma \isom \kreis{\pi}_1(O)/ p_*(\pi_1(X))$,}
where $\kreis{\pi}_1(O)$ denotes the \defi{orbifold fundamental group} of $O$ (see e.g.\ \cite{ScoGeo}) and $p_*$ the canonical projection from $X$ to $O$. Does \eqref{orb pi}, combined with the classification of compact 2-orbifolds, provide information about the groups of  \Tr{tfae}? I do not think so. If $X= \BS^2$ or $X= \R^2$, then $\Gamma$ is one of Maschke's finite groups of isometries of the sphere, or the Fuchsian and Kleinian groups respectively; these cases are already well-understood. The most interesting case is where $X$ is the Cantor sphere $\cc$. It is not too hard to prove that $\Pi:= \pi_1(\cc)$ is the free group of countably infinite rank. Then \eqref{orb pi} tells us that if $\Gamma$ acts faithfully and \pd ly on $\cc$, then  $\Gamma\isom \kreis{\pi}_1(O)/ \Pi$ for some  good compact 2-orbifold. 

Have we learned anything about $\Gamma$? Note that every countable group $G$ can be written as a quotient $F/F'$ where both $F,F'$ are free, because defining \g via a group presentation $\left< \SF \mid \mathcal \RF \right>$ provides such an expression, with $F$ being 
the free group with generating set $\SF$ and $F'$ its smallest normal subgroup spanned by $\RF$. With easy modifications, we may always assume that both $F,F'$ have infinite rank, and are hence isomorphic to $\Pi$. 

Here is a concrete example. Let $\Gamma = \Z$, let $X$ be the open annulus, and let $\Z$ act on $X$ by a shift, so that $O= X/\Z$ is the torus $T$. Then \eqref{orb pi} says that $\Z= \pi_1(T) / \pi_1(X) = \Z^2 / \Z$. However, not every quotient of $\Z^2$ by a normal subgroup $H$ isomorphic to $\Z$ yields a group as in \Tr{tfae}: the quotient $\Z^2 / < (3,3) >$ is a finite abelian group of rank 2, which can act by isometries on the torus but not on $\BS^2$; this can be proved using the techniques of \Sr{sec rels}.

\section{Acting on $\R^3$} \label{sec R3}

\comment{
\begin{theorem}\label{act R3}
Every group $\Gamma$ as in \Tr{tfae} admits a faithful \pd\ action on $\R^3$.
\end{theorem} 
}

The aim of this section is to prove \Tr{act R3}, which states that every group $\Gamma$ as in \Tr{tfae} admits a faithful \pd\ action on $\R^3$. We remark that in the orientation-preserving case \Tr{act R3} immediately follows from \Cr{Kleinian}. The proof below follows a different method.

We will use the notion of planar presentations, and the associated almost planar Cayley complexes, from \cite{planarPresI,planarPresII}, about which we need to prove a couple of additional facts. We start by recalling some terminology.

\medskip
We say that a Cayley complex $Z$ is \emph{almost planar}, if it admits a map $\rho: Z \to \BS^2$ in which the 2-simplices of $Z$ are \emph{nested} in the following sense. We say that  two 2-simplices of $Z$ are nested, if the images of their interiors under $\rho$ are either disjoint, or one is contained in the other, or their intersection is the image of a 2-cell bounded by two parallel edges corresponding to an involution $s$ in the generating set $\SF$ defining $Z$.\footnote{The third option can be dropped by considering the \emph{modified} Cayley complex in the sense of \cite{LyndonSchupp}, i.e.\ by representing involutions in $\SF$ by single, undirected edges.} We call such a $\rho$ an \defi{almost planar map} of $Z$.

Every Cayley complex $Z$ in this section is \defi{finitely presented}, i.e.\ its defining presentation is finite.

The following is proved in \cite[Theorem~5.6]{planarPresII}:\footnote{The difference between this formulation and the one of \cite[Theorem~6.6]{} is only in the terminology used}
\begin{lemma}\label{pla pres}
Let $\g = Cay(\Gamma,\SF)$ be a finitely generated \Cg\ and $\sigma: G \to \BS^2$ a \cova\  embedding. Then \ti\ a (finitely presented) Cayley complex $Z$ of $\Gamma$ with 1-skeleton \G, and an almost planar map $\rho: Z \to \BS^2$ \st\ $\rho(G)$ coincides with $\sigma$. 
\end{lemma} 

We call a Cayley complex \defi{standard}, if the closure of each of its 2-cells is homeomorphic to a disc.
The Cayley complexes provided by \Lr{pla pres} are standard by their construction. 

We let  $Z^1$ denote the 1-skeleton of a Cayley complex $Z$, which is a \Cg\ by the definitions. We write $Z^2$ for the set of 2-cells of $Z$. Given $C\in Z^2$, the boundary $\partial C$ of $C$ is a subgraph of $Z^1$, and if $Z$ is standard then $\partial C$ is always a cycle. 

We call two 2-cells $B,C$ of a Cayley complex $Z$ \defi{equivalent}, if $\partial B = \partial C$. We obtain the corresponding \defi{simplified} Cayley complex by removing all but one representative from each equivalence class of 2-cells of $Z$. 

Let $Z$ be a standard, simplified Cayley complex,  
and let $\rho: Z \to \BS^2$ be an almost planar map. Then every 2-cell $C\in Z^2$ defines two \defi{sides} $C_1,C_2$, namely the two components of $\BS^2 \sm \rho(\partial C)$. Given  $B,C\in Z^2$, we say that $B$ is \defi{maximally nested} in the side $C_1$ of $C$, if $\rho(\partial B) \subset \cls{C_1}$ and there is no 2-cell $B'\neq B$ of $Z$ \st\ $\rho(\partial B')$ separates $\rho(\partial B)$ from $\rho(\partial C)$ in $\BS^2$.  

\begin{lemma}\label{fin nest}
Let $\rho: Z \to \BS^2$ be an almost planar map of a standard, simplified Cayley complex $Z$ \st\ $\rho(Z^1)$ is \cova, let  $C\in Z^2$, and let  $M(C)$ be the set of 2-cells of $Z$ maximally nested in a side $I$ of $C$. Then either  $M(C)$ is finite, or there is exactly one accumulation point of $M(C)$ in $|Z|$.
\end{lemma} 
\begin{proof} 	
Easily, we may assume that \fe\ $D\in Z^2$, $\rho(D)$ is a topological disc bounded by $\rho(\partial D)$. Thus if $\rho(Z^1)$ is given, then $\rho(D)$ is one of the two components of $\BS^2 \sm \rho(\partial D)$. Note that the almost planarity of $\rho$ is not affected if we modify it so as to map $\rho(D)$ to the other component of $\BS^2 \sm \rho(\partial D)$. Therefore, we may assume without loss of generality that 
\labtequ{nest I}{$\rho(D)\subset I$ \fe\ $D\in M(C)$.}

Suppose, to the contrary, that there are distinct ends $\oo,\chi\in \OO(Z)$ in the closure of $\bigcup M(C)$ in $|Z|$.
By \Cr{FG emb k1}, we may assume that $\rho$ extends to an embedding $\rho': |Z| \to \BS^2$. \mymargin{we can make this assumption part of the statement.}
Let $K$ be a cycle in $Z^1$ \st\ $\rho'(K)$ separates $\rho'(\oo)$ from $\rho'(\chi)$, which exists by \Lr{end sep}.
Write its edge-set $E(K)$ as a sum (\wrt\ addition in the cycle space of $Z^1$) $E(K)= \sum_{1\leq i\leq k} E(R_i)$, where each $R_i$ is a cycle  of $Z^1$ induced by a relator in the presentation defining $Z$. (We can choose $R_i$ to be cycle rather than a closed walk because $Z$ is standard.) 

We claim that \fe\ cycle $R$ of $Z^1$ induced by a relator, $\rho'(R)$ does not separate $\rho'(\oo)$ from $\rho'(\chi)$. For if it does, then both components $R_1,R_2$ of $\BS^2 \sm \rho'(R)$ contain infinitely many images of elements of $M(C)$. In particular each of $R_1,R_2$ contains a boundary of an element of $M(C)$ not equal to $R$. This contradicts the fact that the elements of $M(C)$ are maximally nested in $I$, as $\rho'(R)$ separates one of them from $C$.

Combining this claim with \Lr{sep sum} implies that $\rho'(K)$ does not separate $\rho'(\oo)$ from $\rho'(\chi)$, and we have reached a contradiction to the existence of $\oo,\chi$.


\comment{
Suppose $M(C)$ is infinite. Since $Z$ is finitely presented by our assumption, \ti\ an infinite subset $M'\subseteq M(C)$ all elements of which are induced by the same relator. Let $\g=Z^1$ be the 1-skeleton of $Z$.

For  $D\in M(C)$, define the interior $D_I$ of $D$ to be the set of vertices $v$ of \g \st\ $\rho(\partial D)$ separates $\rho(v)$ from  $\rho(\partial C)$ in  in $\BS^2$.  
The set of all other vertices of \g is the \defi{exterior} $D_E$ of $D$. 

By the definition of $M'$, \fe\ $B,D\in M'$ \ti\ an element $g$ of the group $\Gamma$ of $Z$ \st\ $gB=D$, and in particular $g \partial B = \partial D$. Since $\rho(G)$ is \cova, we have either $gB_I= D_I$ or $gB_I= D_E$ whenever $gB=D$. Indeed, the former occurs whenever the action of $g$ on $G$ preserves the cyclic ordering of the edges incident with a vertex, and the latter  occurs whenever  $g$ reverses that ordering.\mymargin{Maybe explain rel'n between consistency and spin.} Thus we can find an infinite subset $M'' \subseteq M'$ \st\ $gB_I= D_I$ \fe\ $B,D\in M''$ and at least one $g=: g^B_D \in \Gamma$.

Pick $B \in M''$, and a \pth{\partial C}{\partial B} $P$ in \G. As $M''$ is infinite, and $\Gamma \act G$ is a free action, we can find $D \in M''$ \st\ $g^B_D (P \cup C) \cap C = \emptyset$. It follows from the choice of $g^B_D$ that $g^B_D P$ lies in $D_E$, and so $g^B_D C$ also lies in $D_E$. On the other hand, $g^B_D C$ lies in $I$ by the choice of $D$. These two facts combined mean that $g^B_D C$ separates $D$ from $C$, contradicting the assumption that $D\in M(C)$ is maximally nested in $I$.
}

\end{proof}

\begin{lemma}\label{nest S2}
Let   $\rho: Z \to \BS^2$ be an almost planar map of a standard, simplified Cayley complex $Z$, \st\ every face boundary of $Z^1$ bounds a 2-cell of $Z$, and $\rho(Z^1)$ extends to an embedding of $|Z^1|$. Let  $I$ be a side of a 2-cell $C$ of $Z$, and let $M$ be the set of 2-cells of $Z$ maximally nested in $I$. If $M$ is finite, then $C \cup \bigcup M$ is homeomorphic to $\BS^2$, and if $M$ is infinite, then $C \cup \bigcup M$ is homeomorphic to $\R^2$.
\end{lemma} 
\begin{proof} 	
Let $H$ be the 1-skeleton of $C \cup \bigcup M$, and notice that $H$ is a subgraph of $Z^1$, and therefore planar.  We claim that 
\labtequ{Hfaces}{the boundary of each face $F$ of $H$ (with respect to the embedding $\rho(H)$) coincides with the boundary of exactly one 2-cell in $\{C\} \cup M$.}
To see this, note first that $\partial C$ bounds a face of $H$ by the definitions.
Let $B\in M$, and suppose $\partial B$ does not bound a face of $H$. This means that there is an edge $e$ of $H$ inside $\rho(B)$, and so \ti\ some $A\in M$ containing $e$ in its boundary. Then $A,B$ will contradict the almost planarity of $\rho$ unless $\rho(\partial B)$ separates $\rho(\partial A)$ from $\rho(\partial C)$. The latter however  
contradicts the assumption that $A$ is maximally nested in $I$. Thus $\partial B$ must bound a face of $H$. 

Conversely, consider a face  $F$ of $H$. If $F$ is a face of $Z^1$ too, then 
$\partial F$ bounds a 2-cell of $Z$ by our assumptions. Otherwise, pick an edge $e\in E(Z^1)$ with $\rho(e)\subset F$. Note that $e$ lies on the boundary of some 2-cell $A$ of $Z$ \iff\ its label appears in at least one of the defining relators. 

If this is the case, let $A'$ be the 2-cell  maximally nested in $I$ \st\ $\rho(A) \subseteq \rho(A')$, which exists because there are only finitely many 2-cells $A'$ with $\rho(A) \subseteq \rho(A') \subseteq I$. Then $A'\in M$, and so $\partial A'$ has no edges in $F$, and $\rho(A')$ meets $F$. This implies that $\partial A' = \partial F$. 

If, on the other hand, $e$ does not lie on the boundary of any 2-cell, then at least one of its endvertices $v$ does not lie on $\partial F$. This is true because an edge of the  \Cg\ $Z^1$ that lies on no relator cycle separates $Z^1$, and so it cannot have both end-vertices on the cycle $\partial F$. In this case, we let  $A$ be a 2-cell of $Z$ incident with $v$ and repeat the above arguments to find $A'\in M$ with $\partial A' = \partial F$. 

Since our Cayley complex is simplified, no other $K\in M$ can satisfy  $\partial K = \partial F$. 

This completes the proof of \eqref{Hfaces}. If $M$ is finite, then $H$ is a finite plane graph, hence attaching a 2-cell along each of its faces yields a homeomorph $S$ of $\BS^2$. As $Z$ is standard, $S$ is  homeomorphic to $C \cup \bigcup M$ as desired. If $M$ is infinite, then $H$ has exactly one end by \Lr{nest S2}.  Recall that we are assuming that $\rho(Z^1)$ extends to an embedding  $\sigma: |Z^1| \to \BS^2$. This induces an embedding  $\sigma': |H| \to  \BS^2$, and attaching, as above, a 2-cell along each of the faces of $\sigma'$ yields a homeomorph $S$ of $\BS^2 \sm \sigma'(\oo) \isom \R^2$, where \oo\ is the unique end of $H$, and again $S$ is homeomorphic to $C \cup \bigcup M$.

\end{proof}

We can now prove the main result of this section.
\begin{proof}[Proof of \Tr{act R3}]
Let \g be a \covaly\ planar \Cg\ of $\Gamma$ with no infinite face-boundaries, as provided by \ref{T iv} of \Tr{tfae}. 
Let $\sigma: |G| \to \BS^2$ be an embedding \st\ $\sigma(G)$ is \cova, provided by \Cr{FG emb k1}.
Let 
$Z'$ be an almost planar Cayley complex with $G= Z^1$, and let $\rho: Z' \to \BS^2$ be an almost planar map, provided by \Lr{pla pres}, \st\ $\rho(G)$ coincides with $\sigma(G)$. 
Recall that we may assume that $Z'$ is standard, and, by adding the corresponding relators to the presentation defining $Z'$ if necessary, that 
\labtequ{ass}{every face boundary of \g \wrt\ $\rho$ bounds a 2-cell of $Z'$.}
Let $Z$ be the corresponding simplified Cayley complex.
 
We now modify the almost planar embedding $\rho(Z)$ of $Z$ into an embedding $\tau: Z \to \D^3$, where $\D^3$ is the closed Euclidean ball of radius 1 in $\R^3$, \st\ $\tau$ coincides with $\rho$ when restricting both maps to $G$, where we think of $\BS^2$ as the boundary of $\D^3$. Let $\kreis{Z}^2$ denote the set of interiors of 2-cells of $Z$.  To define this $\tau$, we just need to specify the image $\tau(C)$ of each $C \in \kreis{Z}^2$. By the definition of the almost planarity of $\rho$, we can easily choose $\tau$ so that $\tau(C) \cap \BS^2 = \emptyset$, and $\tau(C) \cap \tau(D)= \emptyset$ for each two distinct  $C,D \in \kreis{Z}^2$. Indeed, we can define $\tau(C_1), \tau(C_2), \ldots$ inductively for any enumeration \seq{C}\ of $\kreis{Z}^2$, exploiting the fact that the circle $\rho( \partial(C_i) ) \subset \BS^2$ does not cross $\rho( \partial(C_j) )$. Moreover, by making $\tau(C_i)$ sufficiently small (e.g.\ contained in a ball of radius $10 diam( \rho( \partial(C_i)))$ around $\rho( \partial(C_i))$), we can assume that the images $\tau(C_i)$ have no accumulation points in the interior of $\D^3$.

Since $Z$ is just a locally finite 2-complex, it is easy to see that our $\tau$ is an embedding of $Z$ into $\D^3$. Our plan is to extend $Z$, and our action $\Gamma \to Z$, to a 3-complex homeomorphic to $\D^3 \sm \rho(\OO(G))$.
Let us first consider the case where the first alternative of \Lr{fin nest} holds \fe\ $C\in Z^2$, i.e.\ $M(C)$ is always finite. We will later extend our construction to the general case. 

Under this assumption,  \Lr{nest S2} says that \fe\ 2-cell $C$ of $Z$, and each of the two sides $I_1,I_2$ of  $C$, the set $M_i$ of 2-cells  maximally nested in $I_i$ together with $C$ form a homeomorph $S_i= S_i(C)$ of $\BS^2$ unless $\partial C$ bounds a face of \g in the embedding $\sigma$. In the latter case, $M_i$ is empty for one of the sides, $I_1$ say, of $C$. Letting $Z_F$ be the 2-complex obtained from $Z$ by 
adding the faces of  \g as 2-cells (which we can since $\partial F \subset Z^1$ \fe\ face $F$ of \G), we observe that $C \cup F$ is homeomorphic to $\BS^2$ whenever $\partial C$ bounds the face $F$ of \G, and we let $S_1(C)= C \cup F$ in this case. Note that if $D \in Z^2$ lies on the boundary of $S_i(C)$, then $C$ lies on the boundary of $S_{2-i}(D)$ by the definitions. This implies that each $D \in Z^2$ lies on the boundary of exactly two elements of $\cs:= \bigcup_{C\in Z^2} \{S_1,S_2\}$, where we also used \eqref{ass}.

Let $T$ be the 3-complex obtained from $Z_F$ by adding a 3-cell $T(S)$ with simple boundary $S$ \fe\ $S\in \cs$. Then $\tau$ extends into an embedding $\phi$ of $T$ into $\D^3$, because $S$ bounds a homeomorph $S'$ of $\R^3$ in $\D^3$  \fe\ $S\in \cs$ by (a rather easy version of) the generalised Schoenflies theorem \cite{BrownGS, MazurGS}, and we can let $\phi(T(S))= S'$. Easily, $\Gamma \act Z$ extends to a cellular action $\Gamma \act T$. Then $\Gamma \act T$ is (faithful and) \pd, because every cellular action on a CW-complex with finite stabilisers of cells is \pd\ \cite[Theorem~9,~(2)=(10)]{KapPD}.  

Moreover, by the construction of $\phi$, we have $Y:=\phi(T)= \D^3 \sm \sigma(\OO(G))$, and so we can think of $\Gamma \act T$ as an action $\Gamma \act Y$ since $\phi$ is an embedding, and so $Y$ is homeomorphic to $T$ (here we used our assumption that $\tau(C_i)$ have no accumulation points in the interior of $\D^3$). By restricting that action to $\D^3 \sm \BS^2$ we thus obtain a faithful, \pd\ action on the interior of $\D^3$, hence on its homeomorph $\R^3$.

We remark that $\Gamma \act Y$ is co-compact, but its restriction to $\D^3 \sm \BS^2$ is not. 
\medskip

We now consider the general case, where $M(C)$ in \Lr{fin nest} is possibly infinite for some of the 2-cells $C\in Z^2$. For every such $C$, and every side $I_i$ of $C$ for which $M(C)= M_i(C)$ is infinite, \Lr{fin nest} yields a unique end $\oo_i(C)$ accumulating $M_i(C)$. We are going to use  $\oo_i(C)$ in order to triangulate the interior of the sphere formed by $\tau( \bigcup M_i(C))$ and $\rho( \oo_i(C))$. For this, given any edge $e$ in the boundary of an element of $M_i(C)$, we add to $Z_F$ two new 1-cells $e_0,e_1$, each joining a distinct  endvertex of $e$ to $\oo_i(C)$; we also add $\oo_i(C)$ to $Z_F$ as a 0-cell. In addition, we add to $Z_F$ a 2-cell bounded by the triangle $e,e_0,e_1$. We let $W$ denote the 2-complex obtained from $Z_F$ after adding all those cells (for every side $I_i$ of a $C\in Z^2$ with infinite $M_i(C)$).

Note that for every $D\in \{C\} \cup M_i(C)$ where $M_i(C)$ is infinite, the set $N \subset W \sm Z$ of newly added 2-cells sharing an edge with $D$ combined with $D$ forms a homeomorph $S_D$ of $\BS^2$. We define $\cs$ as above, except that we now let $S_i(C) = S_C$ whenever one or both sides $I_1,I_2$ of $C\in Z^2$ has infinite $M_i(C)$. As above, we construct a 3-complex $T$  by adding a 3-cell $T(S)$ with simple boundary $S$ \fe\ $S\in \cs$. Easily, we can extend $\tau$ into an embedding $\tau': W \to \D^3$, and from there  to an embedding $\phi$ of $T$ into $\D^3$, with image $Y=\phi(T)= \D^3 \sm \sigma(\OO(G) \sm \OO')$, where $\OO'$ is the set of 0-cells $\oo_i(C)$ in $W^0 \sm Z_F^0$. The extension $\Gamma \act T$ of our action is still cellular, but it now fails to have finite vertex stabilisers because of the 0-cells in $\OO'$.  In fact, $\Gamma \act T$ is not \pd\ because any point in $\OO'$ accumulates orbits. Therefore, we restrict our action to the topological subspace $T' \subseteq T$ obtained from $T$ by removing $\OO'$.  This $T'$ is not a cell complex anymore, so we need a different argument to prove that $\Gamma \act T'$ is \pd. But this is not hard: we apply \Lr{cover pd} as in the proof of \Lr{pd ext}. The cover \cu\ of $T$ we choose for this comprises the sets $U_v, v\in V(G)$ obtained as the union of all open cells of any dimension in $T'$ that have a vertex $v$ of \g in their boundary. Restricting $\Gamma \act Y$ to $\D^3 \sm \BS^2$ again we obtain the desired action.

\end{proof}

\acknowledgements{I thank Caroline Series for suggesting studying the connection between planar groups and Kleinian groups some years ago. I thank Misha Kapovich for suggesting using Maskit's \Tr{thm Maskit} to prove \Cr{Kleinian}. I thank Matthias Hamann for the proof of \Lr{cutspace}, and Max Pitz for numerous remarks. 

I am very grateful to Brian Bowditch for many helpful discussions.} 

\comment{
\begin{theorem}\label{}
\end{theorem} 
\begin{proof} 	

\end{proof}
}

\bibliographystyle{plain}
\bibliography{../collective}

\begin{thebibliography}{10}

\bibitem{ArzChe}
G.~N. Arzhantseva and P.-A. Cherix.
\newblock On the {Cayley} graph of a generic finitely presented group.
\newblock {\em Bulletin of the Belgian Mathematical Society - Simon Stevin},
  11(4):589--601, December 2004.

\bibitem{BabCon}
L.~Babai.
\newblock Some applications of graph contractions.
\newblock {\em J.~Graph Theory}, 1(2):125--130, 1977.

\bibitem{baGro}
L.~Babai.
\newblock The growth rate of vertex-transitive planar graphs.
\newblock In {\em Proceedings of the eighth {ACM-SIAM} symposium on Discrete
  algorithms (SODA)}, pages 564--573. Soc.\ for Industrial and Applied
  Mathematics, 1997.

\bibitem{BorRie}
J.~E. Borzellino.
\newblock {\em {Riemannian Geometry of Orbifolds, {PhD} thesis}}.
\newblock {University of California, Los Angeles}, 1992.

\bibitem{BowMes4dim}
B.~H. Bowditch and G.~Mess.
\newblock {A 4-Dimensional Kleinian Group}.
\newblock {\em Trans.\ Am.\ Math.\ Soc.}, 344(1):391--405, 1994.

\bibitem{BroPer}
L.~E.~J.\ Brouwer.
\newblock On the structure of perfect sets of points.
\newblock {\em Proc.\ Koninklijke Akademie van Wetenschappen}, 12:785--794,
  1910.

\bibitem{BrownGS}
M.~Brown.
\newblock {A proof of the generalized Schoenflies theorem}.
\newblock {\em Bull. Amer. Math. Soc}, 66:74--76, 1960.

\bibitem{dicks_dunw}
Warren Dicks and Martin~J. Dunwoody.
\newblock {\em Groups acting on graphs}.
\newblock Cambridge University Press, 1989.

\bibitem{diestelBook05}
Reinhard Diestel.
\newblock {\em Graph {T}heory \emph{(3rd edition)}}.
\newblock Springer-Verlag, 2005.
\newblock \\ Electronic edition available at:\\ {\small\tt
  http://www.math.uni-hamburg.de/home/diestel/books/graph.theory}.

\bibitem{DroInf}
C.~Droms.
\newblock {Infinite-ended groups with planar Cayley graphs.}
\newblock {\em J. Group Theory}, 9(4):487--496, 2006.

\bibitem{DrSeSeCon}
C.~Droms, B.~Servatius, and H.~Servatius.
\newblock {Connectivity and planarity of Cayley graphs.}
\newblock {\em Beitr.\ Algebra Geom.}, 39(2):269--282, 1998.

\bibitem{DunAcc}
M.J. Dunwoody.
\newblock The accessibility of finitely presented groups.
\newblock {\em Invent.\ math.}, 81:449--457, 1985.

\bibitem{DunPla}
M.J. Dunwoody.
\newblock Planar graphs and covers.
\newblock Preprint available at {\small http://www.personal.soton.ac.uk/mjd7/},
  2009.

\bibitem{ChGhLeSci}
{Eric Charpentier, Etienne Ghys, Annick Lesne (eds.)}.
\newblock {\em The Scientific Legacy of Poincar\'e}, volume~36 of {\em History
  of mathematics}.
\newblock {American Mathematical Soc}, 2010.

\bibitem{Klein}
{Felix Klein}.
\newblock {\em Vergleichende Betrachtungen \"uber neuere geometrische
  Forschungen}.
\newblock Verlag von Andreas Deichert, Erlangen, 1872.

\bibitem{ltop}
A.~Georgakopoulos.
\newblock Graph topologies induced by edge lengths.
\newblock In {\em {Infinite Graphs: Introductions, Connections, Surveys.
  Special issue of {\it Discrete Math.}}}, volume 311 (15), pages 1523--1542,
  2011.

\bibitem{cayley3}
A.~Georgakopoulos.
\newblock {The planar cubic Cayley graphs}.
\newblock {\em Memoirs of the {AMS}}, 250(1190), 2017.

\bibitem{planarPresI}
A.~Georgakopoulos and M.~Hamann.
\newblock {The planar Cayley graphs are effectively enumerable I: consistently
  planar graphs}.
\newblock To appear in \emph{Combinatorica}.

\bibitem{planarPresII}
A.~Georgakopoulos and M.~Hamann.
\newblock {The planar Cayley graphs are effectively enumerable II}.
\newblock Preprint 2018.

\bibitem{Hatcher}
A.~Hatcher.
\newblock {\em Algebraic {T}opology}.
\newblock Cambrigde Univ.\ Press, 2002.

\bibitem{Hopf}
H.~Hopf.
\newblock {Enden offener R\"aume und unendliche diskontinuierliche Gruppen}.
\newblock {\em Comment.\ Math.\ Helv.}, 16:81--100, 1944.

\bibitem{ImWhi}
W.~Imrich.
\newblock {On Whitney's theorem on the unique embeddability of 3-connected
  planar graphs.}
\newblock In {\em {Recent Adv. Graph Theory, Proc. Symp. Prague 1974}}, pages
  303--306. 1975.

\bibitem{KapPD}
M.~Kapovich.
\newblock A note on properly discontinuous actions.
\newblock ``https://www.math.ucdavis.edu/\textasciitilde
  kapovich/EPR/prop-disc.pdf".

\bibitem{KroInf}
B.~Kr\"on.
\newblock Infinite faces and ends of almost transitive plane graphs.
\newblock Preprint.

\bibitem{Kuratowski}
Kazimierz Kuratowski.
\newblock {\em Topology, Volume II}.
\newblock Academic Press, 1968.

\bibitem{LevMasSpe}
H.~Levinson and B.~Maskit.
\newblock {Special embeddings of Cayley diagrams}.
\newblock {\em J.~Combin.\ Theory (Series B)}, 18(1):12--17, 1975.

\bibitem{LyndonSchupp}
Roger~C. Lyndon and Paul~E. Schupp.
\newblock {\em Combinatorial {Group} {Theory}}.
\newblock Springer Science \& Business Media, January 2001.

\bibitem{MarGeo}
A.~Marden.
\newblock {Geometrically finite Kleinian groups and their deformation spaces}.
\newblock In {\em Discrete groups and automorphic functions, W.~Harvey (ed.)},
  pages 259--293. Academic Press, London, 1977.

\bibitem{Marden}
Albert Marden.
\newblock {\em {Outer Circles}}.
\newblock {Cambridge University Press}, 2007.

\bibitem{Maschke}
H.~Maschke.
\newblock {The representation of finite groups, especially of rotation groups
  of three and four dimensional space, by Cayley's color diagrams}.
\newblock {\em Amer.~J.~Math}, 18:156--194, 1896.

\bibitem{MasCla}
B.~Maskit.
\newblock {On the classification of Kleinian groups: I---Koebe groups}.
\newblock {\em Acta Math.}, 135:249--270, 1975.

\bibitem{Maskit}
Bernard Maskit.
\newblock {\em Kleinian Groups}, volume 287 of {\em Grundlehren der
  mathematischen Wissenschaften}.
\newblock Springer, 1988.

\bibitem{MazurGS}
B.~Mazur.
\newblock {On embeddings of spheres}.
\newblock {\em Bull. Amer. Math. Soc}, 65:59--65, 1959.

\bibitem{mohTre}
B.~Mohar.
\newblock Tree amalgamation of graphs and tessellations of the cantor sphere.
\newblock {\em Journal of Combinatorial Theory. Series B}, 96(5):740--753,
  2006.

\bibitem{MoiseGT}
Edwin~E. Moise.
\newblock {\em {Geometric Topology in Dimensions 2 and 3}}.
\newblock Graduate Texts in Mathematics. Springer, 1977.

\bibitem{Ohshika}
Ken'ichi Ohshika.
\newblock {\em Discrete Groups}, volume 207 of {\em {Iwanami Series In Modern
  Mathematics}}.
\newblock {American Mathematical Society}, 2002.

\bibitem{RadRiem}
T.~Rad\"o.
\newblock {\"Uber den Begriff der Riemannschen Fl\"ache}.
\newblock {\em Acta Sci.\ Math. (Szeged)}, 2(1):96--114, 1925.

\bibitem{Richards}
I.~Richards.
\newblock {ON THE CLASSIFICATION OF NONCOMPACT SURFACES}.
\newblock {\em Trans.\ Am.\ Math.\ Soc.}, 106(2):259--269, 1963.

\bibitem{ThomassenRichter}
R.B. Richter and C.~Thomassen.
\newblock $3$-connected planar spaces uniquely embed in the sphere.
\newblock {\em Trans.\ Am.\ Math.\ Soc.}, 354:4585--4595, 2002.

\bibitem{ScoGeo}
P.~Scott.
\newblock {The Geometries of 3-Manifolds}.
\newblock {\em Bull.\ London Math.\ Soc.}, 15(5):401--487, 1983.

\bibitem{SerCra}
C.~Series.
\newblock A crash course on kleinian groups.
\newblock {\em Rend.\ Istit.\ Mat.\ Univ.\ Trieste}, 37((1-2)):1--38, 2006.

\bibitem{ThoTil}
C.~Thomassen.
\newblock {Tilings of the Torus and the Klein Bottle and Vertex-Transitive
  Graphs on a Fixed Surface}.
\newblock {\em Trans.\ Am.\ Math.\ Soc.}, 323(2):605--635, 1991.

\bibitem{ThoJS}
C.~Thomassen.
\newblock {The Jordan-Schoenflies Theorem and the Classification of Surfaces}.
\newblock {\em American Mathematical Monthly}, 99:116--130, 1992.

\bibitem{ThomassenVellaContinua}
C.~Thomassen and A.~Vella.
\newblock Graph-like continua, augmenting arcs, and {M}enger's theorem.
\newblock {\em Combinatorica}, 29:595--623, 2008.

\bibitem{ThuThr}
W.~P. Thurston.
\newblock {Three dimensional manifolds, Kleinian groups and hyperbolic
  geometry}.
\newblock {\em {Bull.\ Amer.\ Math.\ Soc.\ (N.S.)}}, 6(3):357--381, 1982.

\bibitem{Thurston}
William Thurston.
\newblock {\em The geometry and topology of three-manifolds}.
\newblock Princeton lecture notes, 1980.

\bibitem{TucFin}
T.~W. Tucker.
\newblock {Finite Groups Acting on Surfaces and the Genus of a Group}.
\newblock {\em J.~Combin.\ Theory (Series B)}, 34:82--98, 1983.

\bibitem{whitney_congruent_1932}
H.~Whitney.
\newblock Congruent graphs and the connectivity of graphs.
\newblock {\em American J.\ of Mathematics}, 54(1):150--168, 1932.

\bibitem{ZVC}
H.~Zieschang, E.~Vogt, and H.-D. Coldewey.
\newblock {\em {Surfaces and planar discontinuous groups. Revised and expanded
  transl. from the German by J. Stillwell.}}
\newblock {Lecture Notes in Mathematics 835. Springer-Verlag}, 1980.

\end{thebibliography}

\end{document}